\documentclass[a4paper]{article}
\usepackage[top=2.5cm,bottom=2.5cm,left=2.5cm,right=2.5cm]{geometry} 
\fussy 
\usepackage{amsmath,amssymb,amsthm,cite,graphicx,subfigure,bm,enumerate}
\usepackage{epstopdf}
 \usepackage{float}
\newtheorem{theorem}{Theorem}[section]
\newtheorem{lemma}[theorem]{Lemma}
\newtheorem{proposition}[theorem]{Proposition}
\newtheorem{corollary}[theorem]{Corollary}
\theoremstyle{definition}
\newtheorem{definition}[theorem]{Definition}
\newtheorem{example}[theorem]{Example}

\theoremstyle{remark}
\newtheorem{remark}[theorem]{Remark}
\numberwithin{equation}{section}

\newcommand{\N}{\ensuremath{\mathbb{N}}}

\newcommand{\R}{\ensuremath{\mathbb{R}}}
\newcommand{\RI}{\ensuremath{\mathbb{R}_\mathcal{I}}}

\newcommand{\ol}{\overline}
\newcommand{\sign}[1]{\mathrm{sgn}(#1)}

 \DeclareMathOperator{\dd}{d}

\def\ee{\mathrm{e}}

\begin{document}

\title{\bf New Interval Calculus with Application to Interval Differential Equations}

\author{\\  Wei Liu$^{1}$,  Muhammad Aamir Ali$^{1}$, Yanrong An$^{2,}$\footnote{Corresponding Author} \\\\
\small $^{1}$School of Mathematics, Hohai University, Nanjing 210024,  China\\
\small $^{2}$School of Management, Nanjing University of Posts and Telecommunications,\\
\small  Nanjing 210003, China\\
\small E-mails: liuw626@hhu.edu.cn, mahr.muhammad.aamir@gmail.com, yanrong\_an@163.com    \vspace{7pt}\\
\small \today}
\date{}
 \maketitle

\noindent \textbf{Abstract:\ } This paper presents a systematic study of the calculus of interval-valued functions and its application to interval differential equations.
To this end, first, we introduce new interval arithmetic operations. Under new operations, the space of interval numbers becomes a strict linear space, and indeed a Hilbert space, whereas the traditional interval arithmetic yields only a semilinear space with a defective algebraic structure.
Secondly, by basing derivative and integral of interval-valued functions on the proposed operations, we retain every essential property of classical calculus while seamlessly incorporating ideas from the multiplicative calculus. The resulting unified “hybrid” framework eliminates the tedious case-by-case inspection of switching points required by the gH-derivative, leading to a markedly streamlined computational procedure.
Finally, we establish an existence theorem for solutions of interval differential equations within the new calculus and corroborate its validity and practicality through representative examples. In contrast to the gH-derivative approach, the number of potential solutions does not explode doubly with additional switching points, ensuring robustness in both theory and computation.

\medskip


\medskip


\bigskip

\section{Introduction}

Interval analysis \cite{moo,sai,JKDW01,MG17}, as a mathematical approach for analyzing and handling uncertainty, captures and quantifies the uncertainty in data by representing numerical values in interval form rather than as single point values.
This method has wide applications in many branches of mathematics, including inequalities \cite{CostaTM2019,ZAYL20,ZAYL22}, differential equations \cite{SB09,MTM12,WR22,WangHZ2025},    optimization \cite{HW04,li,GYLZ22,GYT23,PengZY2025}, and decision making \cite{DaiJH2012,LiaoHC2018,YazdaniM2021,QinYJ2025}. In recent years, with the development of artificial intelligence technology, the application scope and influence of interval analysis have been further expanded, especially in engineering \cite{HaqRSU2023,ShehadehA2024},  environment \cite{HaoY2024,ZhuJM2025}, machine  learning \cite{LiWT2023,XuWH20241,MaGZ2025,SanchezL2025}, economics \cite{DragoC2022,RahmanMS2022,WangP2023,WangZZ2025}, and finance \cite{JiangMR2024,WangJJ20241,XiongT2014}, where it is used to handle the inherent imprecision of real-world data that traditional real numbers cannot capture.

In 1966, Moore published the first monograph on interval analysis \cite{moo66}, which made great contributions to the early development of theory and applications of interval analysis.
In this monograph, Moore presented the arithmetic operations of two intervals,  i.e.,
 for two certain intervals $\ol a=[a_l,a_r ]$ and $\ol b=[b_l,b_r ]$, then $\ol a\oplus\ol b=[a_l+  b_l,   a_r+  b_r]$ and
 $\ol a\ominus\ol b=[a_l-  b_r,   a_r-  b_l]$.
 However,    if $\ol a\oplus\ol b=0$, one can not get $\ol a=-\ol b$. This leads to the space of interval numbers being a semilinear space \cite{WR22} rather than a linear space.   Moreover, besides  $\ol a\ominus\ol a \neq 0$,
  $\ol a\oplus\ol b=\ol c$ does not imply that $\ol a=\ol c\ominus\ol b$.
 This indicates that there is a flaw in the operations of intervals.
  In order to establish a theory for interval number as real number, Hukuhara introduced a new concept of interval difference in \cite{MH67}, i.e., $\ol a\ominus_H\ol b=[a_l-  b_l,   a_r-  b_r]$. It is easy to see that the Hukuhara subtraction  is the
 inverse operation of addition. However, the Hukuhara difference  of two intervals does not always exist. For example, $[1,2]\ominus_H[3,5]=[-2,-3]$, the result does  not make sense.  In 1979, Markov \cite{M79} presented a new difference  of two intervals, which is consistent with the concept later called the generalized Hukuhara difference (gH-difference) by Stefanini in \cite{Stefanini2008}, i.e.,
 $$\ol a\ominus_{gH}\ol b=[\min\{a_l-  b_l,   a_r-  b_r\},\max\{a_l-  b_l,   a_r-  b_r\}].$$
 The gH-difference always exists, and  it is the inverse operation of addition in specific situations.      So  it has been widely used in various papers \cite{SB09,LS10,Chalco2013,VL15,HLO17,HVN18,CMDJ19,WR22}.

 On the other hand, based on the arithmetic operations of intervals, many researchers have focused on the calculus of interval-valued functions with application to interval differential equations. For example, Moore \cite{moo} introduced the
  integration of interval
functions, with an introduction to automatic differentiation,  and used it to  treat integral and differential equations.
Based on the Hukuhara difference,
   Wu  and Gong \cite{WG2000} defined the Henstock integral of interval-valued functions, and studied the necessary and sufficient conditions of integrability and characterization of integrable interval-valued functions.

In \cite{SB09}, based on the gH-difference,
 Stefanini and Bede introduced  two definitions for the derivative of an interval-valued function and their properties.
Moreover, local existence and uniqueness of
two solutions was obtained together with characterizations of the solutions of an interval
differential equation by ODE systems.
Later in \cite{MTM12}, an interval initial value problem with a second type Hukuhara derivative was considered.
The existence of solutions,   continuous dependence of the solution on initial value, and the compactness of solutions set  were shown.
In \cite{Chalco2013},  conditions, examples and counterexamples for limit, continuity, integrability, and differentiability for interval-valued functions were given by using the generalized Hukuhara differentiability.
In \cite{WR22},  some sufﬁcient conditions are provided in order to deduce the existence of solutions without switching points, and also for mixed solutions with a unique switching point, while stopping time problem of the ﬁrst order multidimensional interval-valued dynamic system was discussed in \cite{WangHZ2025}. These advancements have not only enriched the theoretical foundations but also broadened the scope of applications, particularly in fuzzy systems and stochastic analysis.
However, due to the existence of two types of derivative, it is necessary to consider the problem of switching points during the differentiation process \cite{Qiu23}. Furthermore, when solving differential equations, as the domain expands, the number of solutions to the equation may increase doubly.

After observing the shortcomings of previous literatures in interval analysis, particularly the algebraic deficiencies of traditional interval arithmetic operations, this work emerges to address several critical limitations.
The conventional approach renders the space of interval numbers $\RI$ merely a semilinear space, creating substantial obstacles in developing a consistent calculus for interval-valued functions. These structural weaknesses manifest most prominently in differentiation and integration theories, where existing definitions lead to operational inconsistencies and computational complexities. This work aims to contribute to this growing body of knowledge by addressing key challenges and exploring new directions in the calculus of interval-valued functions.

Motivated by these challenges, we introduce novel arithmetic operations that restore complete linearity to $\RI$, transforming it into a proper Hilbert space equipped with a well-defined inner product ``$\diamond$" and norm $\|\cdot\|$. This enhanced structure enables us to propose:
\begin{itemize}
    \item A new derivative concept unifying classical and multiplicative approaches;
    \item An integral definition overcoming previous computational limitations;
    \item A rigorous framework for interval differential equations.
\end{itemize}

The proposed calculus demonstrates superior computational efficiency compared to existing methods, while maintaining intuitive interpretability for applications ranging from uncertainty quantification to growth process modeling. Our results not only resolve long-standing theoretical inconsistencies, but also provide practical tools for working with interval-valued functions in scientific computing and engineering applications.

The rest of the paper are organized as follows.
In Section 2, we present  new interval arithmetic operations, which lead to the space of all interval numbers being a linear space. Further, by introducing new distance, norm, and inner product formulas, we prove that the space of all interval numbers is a Hilbert space.
Section 3 introduces the interval-valued function and its continuity. Under the new norm, the space of continuous interval-valued functions is a Banach space.
Section 4 mainly introduces the new concept of the derivative of interval-valued functions and its related properties. The new derivative is a combination of the traditional derivative and the multiplicative derivative, which simplifies the differentiation of interval-valued functions and eliminates the need to consider the switching points. Compared to the gH-derivative, our derivative has  advantages in  computation.
In Section 5, we consider the Riemann-type integral of interval-valued functions and present many related properties, including linearity, monotonicity, subinterval integrability, and  additivity. At the same time, the fundamental theorem of calculus, integration by parts formula, monotone convergence theorem, and dominated convergence theorem are also proved.
In Section 6, by using the well-known Schauder's fixed point theorem, we present an existence theorem for the solutions of interval differential equations and demonstrate the effectiveness and practicality of our method through several specific examples.

\section{The Space $\RI$}

Let $\R$ be the set of real numbers, and $  a_l,  a_r\in\R$ with $  a_l<   a_r$, the  closed interval $\ol a=[ a_l,  a_r]$ is called an  {\it   interval number}.
 Denote  by  $\RI$ the set of all   interval numbers, i.e.,
 $$\RI=\{\ol a~|~\ol  a=[ a_l,  a_r],~ a_l<   a_r\}.$$
{\bf In this paper,   real numbers are NOT considered as  (degenerate) interval numbers.}

For any interval number $\ol a=[  a_l,  a_r]\in \RI$, let
\begin{equation}\label{eq1}
a_c=\frac{a_l+a_r}{2},\qquad
a_w=\frac{  a_r- a_l}{2}.
\end{equation}
Then $a_c$ $(\in \R)$ is called the \textit{center} of $\ol a$ and $a_w$ $(>0)$ is called the \textit{radius} of $\ol a$. In fact, for any interval number, its  center and radius are unique. In this sense, we can rewrite the interval number $\ol a=[a_l,a_r]$ by
$$\ol a\triangleq \langle a_c;a_w\rangle.$$
This  notation is also shown in many papers, such as
\cite{BS14,Stefanini19,Stefanini192}.

\subsection{Orderings}
This subsection is devoted to some common used orderings in $\RI$.

\begin{definition}\label{TOrder}(\cite{BS14})
For any $\ol a,\ol b\in \RI,$ $\ol a = \langle a_c;a_w\rangle$, $\ol b= \langle b_c;b_w\rangle,$ we say that
\begin{enumerate}[(1)]
 \item $\ol a=\ol b$  if   $a_c=b_c$ and $a_w=b_w$.

\item $\ol a<\ol b$  if  $a_c< b_c$ or $a_c=b_c$ and $a_w<b_w$.

\item $\ol a\leq\ol  b$  if  $\ol a< \ol b$ or $\ol a=\ol b$.
\end{enumerate}
\end{definition}

Let $\ol a,~\ol b\in \RI$, $\ol a = \langle a_c;a_w\rangle$, $\ol b= \langle b_c;b_w\rangle,$
the function $\phi:\RI\times \RI\to \R$ is given by
\begin{equation}\label{phi}
  \phi(\ol a,\ol b)=(a_c-b_c)+\left(1-{\sign {|a_c-b_c|}}\right)\left(\frac{a_w}{b_w}-1\right),
\end{equation}
where  $\sign{.}$ is the signum function.
\begin{lemma} Assume that   $\ol a,~\ol b\in \RI$, then
   $\ol a\leq\ol  b$ if and only if $\phi(\ol a,\ol b)\leq 0.$
\end{lemma}
\begin{proof}
  (Sufficiency)  Since $\phi(\ol a,\ol b)\leq 0,$
  if $a_c\neq b_c$, then
  $$\phi(\ol a,\ol b)=a_c-b_c\leq 0,$$
  hence,
  $a_c<b_c,$
  which implies $\ol a<\ol b$;
if $a_c=b_c$, then
$$\phi(\ol a,\ol b)=\frac{a_w}{b_w}-1\leq 0,$$
  i.e.,
  $a_w\leq b_w,$
  which implies $\ol a\leq \ol b$.

  (Necessity)  If $\ol a=\ol b$, then
  $a_c=b_c$ and $a_w=b_w$. Therefore
  $$\phi(\ol a,\ol b)= 0.$$
   If $\ol a<\ol b$, then $a_c<b_c$ or $a_c=b_c$ and $a_w<b_w$. When $a_c<b_c$, one has
  $$\phi(\ol a,\ol b)=a_c-b_c< 0.$$
  When $a_c=b_c$ and $a_w<b_w$, one has
  $$\phi(\ol a,\ol b)=\frac{a_w}{b_w}-1< 0.$$
  The proof is therefore complete.
\end{proof}

\begin{example}
Consider the interval  numbers $\ol a=[0,3],~\ol b=[1,2],~\ol c=[3,4].$ Then we have
 $$\phi(a,b)=2>0,\qquad \phi(a,c)=-2<0.$$
Thus, $b<a<c$.
\end{example}

\begin{definition}\label{defpreceq}(\cite{BS14})
  For any $\ol a,\ol b\in \RI$, $\ol a = \langle a_c;a_w\rangle$, $\ol b= \langle b_c;b_w\rangle,$ we define
\begin{itemize}
    \item  $\ol a\preccurlyeq \ol b$ if $a_c\leq  b_c$   and $ a_w\leq  b_w$.
   \item  $\ol a\prec\ol b$ if $\ol a\preccurlyeq \ol b$ but $\ol a\neq  \ol b$.
\end{itemize}
\end{definition}

By  Definitions \ref{TOrder} and  \ref{defpreceq}, we have
\begin{lemma}\label{lem2.8}
  If $\ol a\preccurlyeq \ol b$ then $\ol a\leq  \ol b$, but not vice versa.
\end{lemma}

\begin{definition}\label{defsubseteq}(\cite{ZDF19})
  For any $\ol a,\ol b\in \RI$,  $\ol a = \langle a_c;a_w\rangle$, $\ol b= \langle b_c;b_w\rangle,$ we define
\begin{itemize}
  \item  $\ol a\subseteq\ol b$ if $b_l\leq  a_l<a_r\leq b_r$ .
  \item  $\ol a\subset\ol b$ if  $\ol a\subseteq\ol b$ but $\ol a\neq \ol b$.
\end{itemize}
\end{definition}
\begin{remark} From Definitions \ref{TOrder}, \ref{defpreceq}, and  \ref{defsubseteq}, it is easy to see that $``\leq"$ is a total ordering, while $``\subseteq"$ and $``\preccurlyeq"$ are partial orderings.
\end{remark}

By Definitions \ref{TOrder}, \ref{defpreceq}, \ref{defsubseteq}, and \eqref{eq1}, one has
\begin{proposition} The following statements are equivalent:
\begin{enumerate}[(1)]
  \item  $\ol a=\ol b$;
  \item  $a_c=b_c$ and  $  a_w=  b_w$;
  \item  $  a_l=   b_l$ and  $  a_r = b_r$;
  \item  $\ol a\leq \ol b$ and  $\ol b\leq \ol a$.
    \item  $\ol a\subseteq \ol b$ and  $\ol b\subseteq \ol a$.
            \item  $\ol a\preccurlyeq \ol b$ and  $\ol b\preccurlyeq \ol a$.
\end{enumerate}
\end{proposition}

\subsection{Arithmetic Operations}
In this subsection, we introduce new arithmetic operations of interval numbers, and make a comparison with the classical interval operations.

\subsubsection{Addition}
\begin{definition}\label{d2}
For any $\ol a,\ol b\in \RI$,  we define
\begin{equation} \label{eq2-1}
\ol a+\ol b=\langle  a_c+  b_c;  a_w   b_w\rangle.
\end{equation}
\end{definition}

\begin{lemma}
 For any $\ol a,\ol b\in \RI$,  there exists a unique $\ol c= [ c_l,  c_r]$ in $\RI$ such that $\ol c=\ol a+\ol b$, and
\begin{align*} 
c_l=\frac{2(a_l+a_r+b_l+b_r)-(a_r-a_l)(b_r-b_l)}{4},\\
c_r=\frac{2(a_l+a_r+b_l+b_r)+(a_r-a_l)(b_r-b_l)}{4}.
\end{align*}
\end{lemma}
\begin{proof} By Definition \ref{d2}, we have
  \begin{align*}
&\ol a+\ol b\\
=&\langle  a_c+  b_c;  a_w   b_w\rangle\\
=&[a_c+  b_c-  a_w   b_w, a_c+  b_c+ a_w   b_w]\\
 =&\left[\frac{a_l+a_r}{2}+\frac{b_l+b_r}{2}-\frac{a_r-a_l}{2}\frac{b_r-b_l}{2},\frac{a_l+a_r}{2}+\frac{b_l+b_r}{2}+\frac{a_r-a_l}{2}\frac{b_r-b_l}{2}\right]\\
 =&\left[\frac{2(a_l+a_r+b_l+b_r)-(a_r-a_l)(b_r-b_l)}{4},\frac{2(a_l+a_r+b_l+b_r)+(a_r-a_l)(b_r-b_l)}{4}\right].
\end{align*}
Let
\begin{align*} 
c_l=\frac{2(a_l+a_r+b_l+b_r)-(a_r-a_l)(b_r-b_l)}{4},\\
c_r=\frac{2(a_l+a_r+b_l+b_r)+(a_r-a_l)(b_r-b_l)}{4},
\end{align*}
then we have $\ol c=\ol a+\ol b$.
\end{proof}

\begin{proposition}\label{p1}
 For any   $\ol a,~\ol b,~\ol c,~\ol d\in \RI,$  one has
\begin{enumerate}[(1)]
\item  $\ol a+\ol b=\ol b+\ol a;$

\item $\ol a+(\ol b+\ol c)=(\ol a+\ol b)+\ol c;$

\item there exists a  unique {\it zero element}  $\ol0 =\langle    0;  1\rangle=[-1,1] $ such that $\ol a+\ol 0=\ol a;$

\item there exists a  unique $\ol e\in \RI$ such that $\ol a+\ol e=\ol 0,$  here $\ol e $ is called the {\it negative element} of $\ol a$ and denoted by  $\ol e=-\ol a$;

\item $\ol a+\ol c\leq \ol b+\ol c$ if and only if $\ol a\leq \ol b$;

\item  $\ol a\leq \ol b$ and $\ol c\leq \ol d$ implies $\ol a+\ol c\leq \ol b+\ol d$.
\end{enumerate}
\end{proposition}
\begin{proof}(1) $\ol a+\ol b=\langle  a_c+  b_c;  a_w   b_w\rangle=\langle  b_c+  a_c;  b_w   a_w\rangle=\ol b+\ol a$.

(2) $\ol a+(\ol b+\ol c)=\langle  a_c; a_w\rangle+  \langle b_c+  c_c;  b_w   c_w\rangle=\langle a_c+b_c+  c_c; a_w b_w   c_w\rangle.$

$(\ol a+\ol b)+\ol c=\langle   a_c+  b_c;  a_w   b_w\rangle+  \langle   c_c;   c_w\rangle=\langle a_c+b_c+  c_c; a_w b_w   c_w\rangle.$

Hence,  $\ol a+(\ol b+\ol c)=(\ol a+\ol b)+\ol c.$

(3) Let $\ol 0=\langle    0_c;  0_w \rangle,$ then we have
$$\ol a+\ol 0=\langle  a_c+  0_c;  a_w0_w \rangle=\langle  a_c;  a_w \rangle=\ol a.$$
This implies  $0_c=0 $ and $0_w=1$, i.e., $\ol0 =\langle    0;  1\rangle=[-1,1].$
Hence, $\ol0 =[-1,1]  $ is the zero element in $\RI$, and the uniqueness is obvious.

(4) Let $\ol e=\langle e_c; e_w\rangle$ such that $\ol a+\ol e=\ol 0.$ Then, according to Definition \ref{d2}, we have
$$a_c+e_c=0,~~a_we_w=1.$$
Thus, $$\ol e=-\ol a=\left\langle-a_c; \frac{1}{a_w}\right\rangle.$$
Hence, for any number $\ol a$, the negative element $-\ol a=\langle-a_c; \frac{1}{a_w}\rangle$, and the uniqueness is obvious.

(5) Since
\begin{align*}
 & \phi(\ol a+\ol c,\ol b+\ol c)\\=&((a_c+c_c)-(b_c+c_c))+\left(1-{\sign {|(a_c+c_c)-(b_c+c_c)|}}\right)\left(\frac{a_wc_w}{b_wc_w}-1\right),\\
  =&(a_c-b_c))+\left(1-{\sign {|a_c-b_c|}}\right)\left(\frac{a_w}{b_w}-1\right),\\
  =& \phi(\ol a,\ol b),
\end{align*}
Thus, by Theorem  \ref{phi}, $\ol a+\ol c\leq \ol b+\ol c$ if and only if $\ol a\leq \ol b$.

(6) Since $\ol a\leq\ol  b$ and $\ol c\leq\ol  d$, according to Definition \ref{TOrder}, we have
\begin{align*}
a_c<b_c, \text{   or  } a_c=b_c,~0<a_w\leq b_w,
\end{align*}
{and}
\begin{align*}
c_c<d_c, \text{   or  } c_c=d_c,~0<c_w\leq d_w.
\end{align*}
It is easy to obtain
$$\left\{\begin{array}{ll}
a_c+c_c=b_c+d_c,~a_wc_w\leq b_wd_w,~\text{or}\\
a_c+c_c<b_c+d_c,
\end{array}\right.$$
that is,
$$\ol a+\ol c\leq \ol b+\ol d .$$
Therefore, we complete the proof.
\end{proof}

Recall that,  in \cite{moo66} and the references therein, the sum is defined as
\begin{align}\label{eq2-11}
  \ol a\oplus\ol b=[a_l+b_l, a_r+b_r]=\langle a_c+b_c;a_w+b_w\rangle.
\end{align}

From \eqref{eq2-1} and \eqref{eq2-11}, it is easy to see that the sums $\ol a+\ol b$ and $\ol a\oplus\ol b$ have the same center, but their radiuses   depend on the  radiuses of $\ol a$ and $\ol b$.
Moreover, we have

\begin{proposition}\label{prob2.12}
 For any $\ol a,\ol b\in \RI$, Let $\ol c=\ol a+\ol b$, and  $\ol d=\ol a\oplus\ol b$.
\begin{enumerate}[(1)]
  \item If $\frac{1}{a_w}+\frac{1}{b_w}=1$, then  $\ol c= \ol d;$
  \item  If $\frac{1}{a_w}+\frac{1}{b_w}<1$, then   $\ol d\subset \ol c;$
  \item  If $\frac{1}{a_w}+\frac{1}{b_w}>1$, then $\ol c\subset \ol d.$
\end{enumerate}
\end{proposition}

\begin{proof} (1) If $\frac{1}{a_w}+\frac{1}{b_w}=1$, then we have $a_w+b_w=a_w b_w$, which implies
$$\langle a_c+b_c;a_wb_w\rangle=\langle a_c+b_c;a_w+b_w\rangle,$$
that is, $$\ol c=\ol a+\ol b=\ol a\oplus\ol b= \ol d.$$

The proofs of (2) and (3) are similar to that of (1).
\end{proof}

\begin{example} Considering that $\ol a=[-5,-1]$, $\ol b_1=[1,3]$, $\ol b_2=[1,5]$, $\ol b_3=[1,7]$, then
  $$\ol a+\ol b_1=\left[-3,1\right],\quad\ol a+\ol b_2=\left[-4,4\right],\quad \ol a+\ol b_3=\left[-5,7\right], $$
  and
   $$\ol a\oplus\ol b_1=\left[-4,2\right],\quad \ol a\oplus\ol b_2=\left[-4,4\right],\quad \ol a\oplus\ol b_3=\left[-4,6\right].$$
\end{example}

\subsubsection{Scalar Multiplication}
\begin{definition}\label{d21}
For any $\ol a\in \RI$ and $k\in \R$,  we define
\begin{equation}\label{eq3}
k\cdot\ol  a=k \ol a=\langle k  a_c;  a_w^k\rangle.
\end{equation}
\end{definition}

\begin{lemma}
 For any $\ol a\in \RI$ and $k\in\R$,  there exists a unique $\ol c= [  c_l,  c_r]$ in $\RI$ such that $\ol c=k \ol a $, and
\begin{align*} 
c_l= ka_c-a_w^k,\qquad
c_r= ka_c+a_w^k.
\end{align*}
\end{lemma}
 \begin{proof}
  This conclusion can be directly gotten from Definition \ref{d21}.
 \end{proof}
\begin{proposition}\label{p2}   For any  $\ol a,\ol b\in \RI $ and $k,l\in\R$,  then
\begin{enumerate}[(1)]
  \item     $ 1\ol a=\ol a;$

 \item $k(\ol a+\ol b)=k\ol a+k\ol b;$

  \item $(k+l)\ol a=k\ol a+l\ol a;$

  \item $(kl)\ol a=k(l\ol a).$

\item $\ol a\leq\ol  b$ implies $k\ol a\leq k\ol b$ if $k\geq 0$, or $k\ol a\geq k\ol b$ if $k< 0$.
\end{enumerate}

\end{proposition}
\begin{proof}

(1) $1\ol a=\langle 1    a_c;   a_w^1\rangle=\langle      a_c;   a_w \rangle=\ol a$.

(2) Since
$$k(\ol a+\ol b)=\langle k    (a_c+b_c);   (a_wb_w)^k\rangle$$
and
$$k\ol a+k\ol b=\langle k  a_c;    a_w^k\rangle+\langle k  b_c;    b_w^k\rangle=\langle k  a_c+kb_c;    a_w^kb_w^k\rangle, $$
then
$$k(\ol a+\ol b)=k\ol a+k\ol b.$$

(3) Since
$$(k+l)\ol a=\langle(k+l)  a_c;  a_w^{k+l }\rangle,$$
and
$$k\ol a+l\ol a=\langle k  a_c;  a_w^k\rangle+\langle  l  a_c;  a_w^l\rangle=\langle(k+l)  a_c;  a_w^{k+l}\rangle,$$
then
$$(k+l)\ol a=k\ol a+l\ol a.$$

(4) Since
$$(kl)\ol a=\langle kl  a_c;  a_w^{kl}\rangle,$$
and
$$k(la)=k\langle l  a_c;   a_w^l\rangle=\langle kl   a_c;    a_w^{kl}\rangle,$$
then
$$(kl)\ol a=k(l\ol a).$$

(5) Since $\ol a\leq\ol b$, according to Definition \ref{TOrder}, we have
$$\left\{\begin{array}{ll}
a_c=b_c,~a_w\leq b_w,~\text{or}\\
a_c<b_c.
\end{array}\right.$$

If $k\geq 0$, then we have
$$\left\{\begin{array}{ll}
ka_c=kb_c,~a_w^k\leq b_w^k,~\text{or}\\
ka_c< kb_c,
\end{array}\right.$$
which implies $k\ol a\leq k\ol b$.

If $k<0$, then we have
$$\left\{\begin{array}{ll}
ka_c=kb_c,~a_w^k\geq b_w^k,~\text{or}\\
ka_c>kb_c,
\end{array}\right.$$
which implies $k\ol a\geq k\ol b$.

Thus, we complete the proof.
\end{proof}

\begin{remark}\label{rem1} By Proposition \ref{p1} (4) and Definition \ref{d21}, when $k=-1$, we have $-\ol a =(-1)\ol  a$. In fact, for any    interval number $\ol a=\langle  a_c;  a_w\rangle$,   one has
$$-\ol a=\langle - a_c;  a_w^{-1}\rangle=-1 \langle  a_c;  a_w\rangle=-1 \ol a.$$
\end{remark}

  \begin{proposition} Let $\ol a_i=\left\langle (a_i)_c;(a_i)_w\right\rangle\in\RI$, $k_i\in\R$, $i=1,2,...,n$. Then
  $$\sum_{i=1}^{n}k_i\ol a_i=\left\langle\sum_{i=1}^{n}k_i  (a_i)_c;\prod_{i=1}^{n}(a_i)_w^{k_i}\right\rangle.$$
  \end{proposition}

Recall that,  in \cite{moo66} and the references therein, the scalar multiplication is defined as
\begin{align}\label{eq3-11}
  k\odot\ol a =\begin{cases}
                 [k a_l, ka_r]=\langle ka_c;ka_w\rangle, &\text{ if } k\geq0,\\
                [k a_r, ka_l]=\langle ka_c;-ka_w\rangle, &\text{ if } k<0 .
               \end{cases}
\end{align}

  According to  \eqref{eq3} and \eqref{eq3-11}, the scalar multiplication $\ol c$ and $\ol d$ have the same center, but their radiuses   depend on the  radius  of $\ol a$ and $k$. Moreover, one has
\begin{proposition}\label{prop2.20}
 For any $\ol a\in \RI$, $k\in\R$, Let $\ol c=k\ol a $, and  $\ol d=k\odot\ol a$.
\begin{itemize}
  \item[(1)] If $|k|a_w=a_w^k$, then  $\ol c= \ol d;$
  \item[(2)] If $|k|a_w<a_w^k$, then   $\ol d\subset \ol c;$
  \item[(3)] If $|k|a_w>a_w^k$, then $\ol c\subset \ol d.$
\end{itemize}
\end{proposition}

\begin{proof} (1) If $|k|{a_w} ={a_w}^k$, then we have
$$\langle ka_c ;a_w^k\rangle=\langle ka_c;|k|a_w\rangle=\begin{cases}
  \langle ka_c;ka_w\rangle, &\text{ if } k\geq0,\\
  \langle ka_c;-ka_w\rangle, &\text{ if } k<0,\\
\end{cases}$$
that is, $$\ol c=k\ol a=k\odot\ol a= \ol d.$$

The proofs of (2) and (3) are similar to that of (1).
\end{proof}

\begin{example}\label{exSM} Considering that $\ol a_1=[-2,-1]$,   $\ol a_2=[1,5]$,   $\ol a_3=[1,9]$, then
$$k\ol a_1=\begin{cases}
[-\frac12,\frac72],  &\text{ if } k=-1,\\
[-1,1],  &\text{ if } k=0,\\
[-\frac34-\frac{\sqrt{2}}{2},-\frac34+\frac{\sqrt{2}}{2}],  &\text{ if } k=\frac12,\\
[-2,-1],  &\text{ if } k=1,\\
[-\frac{13}{4},-\frac{11}{4}],  &\text{ if } k=2,\\
\end{cases}\qquad\text{~and~}
 k\odot\ol a_1=\begin{cases}
[1,2], &\text{ if } k=-1,\\
0,  &\text{ if } k=0,\\
[-1,-\frac12],  &\text{ if } k=\frac12,\\
[-2,-1],  &\text{ if } k=1,\\
[-4,-2],  &\text{ if } k=2,\\
\end{cases}
$$
$$k\ol  a_2=\begin{cases}
[-\frac72,-\frac52],  &\text{ if } k=-1,\\
[-1,1],  &\text{ if } k=0,\\
[1,5],  &\text{ if } k=1,\\
[\frac{9}{2}-2\sqrt{2},\frac{9}{2}+2\sqrt{2}],  &\text{ if } k=1.5,\\
[2,10],  &\text{ if } k=2,\\
[1,17],  &\text{ if } k=3,\\
\end{cases}\qquad\text{~and~}
 k\odot\ol a_2=\begin{cases}
[-5,-1], &\text{ if } k=-1,\\
0,  &\text{ if } k=0,\\
[1,5],  &\text{ if } k=1,\\
[\frac32,\frac{15}{2}],  &\text{ if } k=1.5,\\
[2,10],  &\text{ if } k=2,\\
[3,15],  &\text{ if } k=3,\\
\end{cases}
$$
$$k\ol a_3=\begin{cases}
[-3,-2],  &\text{ if } k=-\frac12,\\
[-1,1],  &\text{ if } k=0,\\
[\frac12,\frac92],  &\text{ if } k=\frac12,\\
[\frac{15}{4}-2\sqrt{2},\frac{15}{4}+2\sqrt{2}],  &\text{ if } k=\frac34,\\
[1,9],  &\text{ if } k=1,\\
[-6,26],  &\text{ if } k=2,\\
\end{cases}\qquad\text{~and~}
 k\odot \ol a_3=\begin{cases}
[-\frac92,-\frac12],  &\text{ if } k=-\frac12,\\
0,  &\text{ if } k=0,\\
[\frac12,\frac92],  &\text{ if } k=\frac12,\\
[\frac34,\frac{27}{4}],  &\text{ if } k=\frac34,\\
[1,9],  &\text{ if } k=1,\\
[2,18],  &\text{ if } k=2.\\
\end{cases}
$$
\end{example}

\subsubsection{Subtraction}
\begin{definition}\label{subt}For any $\ol a,~\ol b\in \RI$,  we define
\begin{align}\label{eq4-110}
  \ol a-\ol b= \left\langle a_c-b_c; \frac{a_w}{b_w}\right\rangle=\left[a_c-b_c-\frac{a_w}{b_w},a_c-b_c+\frac{a_w}{b_w}\right].
\end{align}
\end{definition}
It follows from Definition \ref{subt} that for any   $\ol a\in \RI$, $$\ol a-\ol a=\left\langle 0; 1\right\rangle=\ol 0.$$

By Definition \ref{subt} and Remark \ref{rem1}, the following statement holds.
\begin{lemma}
  For any $\ol a,~\ol b\in \RI$,  we have
   $$\ol a-\ol b= \ol a+(-\ol b)=\ol a+(-1\ol b).$$
\end{lemma}

\begin{theorem}\label{thm331}
  Let $\ol a,~\ol b,~\ol c\in \RI$. If $\ol a+\ol b=\ol c$, then $ \ol a=\ol c-\ol b$.
\end{theorem}
\begin{proof}
  Since $\ol a+\ol b=\ol c$, we have
  $$c_c=a_c+b_c,\qquad c_w=a_wb_w.$$
  Thus
  $$a_c=c_c-b_c,\qquad a_w=\frac{c_w}{b_w}.$$
That is,
$$\ol a=\ol c-\ol b.$$
This ends the proof.
\end{proof}

\begin{remark}
  Theorem \ref{thm331} implies that the  subtraction defined in Definition \ref{subt} is the inverse arithmetic   of the addition given Definition \ref{d2}.
\end{remark}

\begin{proposition} For any $\ol a,~\ol b,~\ol c\in \RI$,  one has
\begin{itemize}
  \item[(1)] $\ol a+(\ol b-\ol c)=\ol a+\ol b-\ol c=\ol a-\ol c+\ol b=\ol a-(\ol c-\ol b)$;
  \item[(2)] $\ol a-(\ol b+\ol c)=\ol a- \ol b-\ol c;$
  \item[(3)] $k(\ol a-\ol b)=k \ol a-k\ol b,$ for any $k\in\R$;
  \item[(4)] $(k-l)\ol a=k \ol a-l\ol a,$ for any $k,l\in\R$.
\end{itemize}

\end{proposition}

\begin{proof}
  (1) For any  $\ol a,~\ol b,~\ol c\in \RI$, one has
  $$\ol a+(\ol b-\ol c)=\left\langle a_c; {a_w}\right\rangle+\left\langle b_c-c_c; \frac{b_w}{c_w}\right\rangle=\left\langle a_c+b_c-c_c; \frac{a_wb_w}{c_w}\right\rangle,$$
   $$\ol a+\ol b-\ol c=\left\langle a_c+b_c; {a_w}{b_w}\right\rangle-\left\langle c_c; {c_w}\right\rangle=\left\langle a_c+b_c-c_c; \frac{a_wb_w}{c_w}\right\rangle,$$
    $$\ol a-\ol c+\ol b=\left\langle a_c-c_c; \frac{a_w}{c_w}\right\rangle+\left\langle b_c; {b_w}\right\rangle=\left\langle a_c+b_c-c_c; \frac{a_wb_w}{c_w}\right\rangle,$$
     $$\ol a-(\ol c-\ol b)=\left\langle a_c; {a_w}\right\rangle-\left\langle c_c-b_c; \frac{c_w}{b_w}\right\rangle=\left\langle a_c+b_c-c_c; \frac{a_wb_w}{c_w}\right\rangle.$$
     Hence,
  $$\ol a+(\ol b-\ol c)=\ol a+\ol b-\ol c=\ol a-\ol c+\ol b=\ol a-(\ol c-\ol b).$$
 (2)  For any  $\ol a,~\ol b,~\ol c\in \RI$, one has
 $$\ol a-(\ol b+\ol c)=\left\langle a_c; {a_w}\right\rangle+\left\langle b_c+c_c; {b_w}{c_w}\right\rangle=\left\langle a_c-b_c-c_c; \frac{a_w}{b_wc_w}\right\rangle,$$
 $$\ol a- \ol b-\ol c=\left\langle a_c-b_c; \frac{a_w}{b_w}\right\rangle-\left\langle c_c; {c_w}\right\rangle=\left\langle a_c-b_c-c_c; \frac{a_w}{b_wc_w}\right\rangle.$$
  Hence, $$\ol a-(\ol b+\ol c)=\ol a- \ol b-\ol c.$$
  (3)  For any  $\ol a,~\ol b\in \RI$, $k\in \R$,  we have
  $$k(\ol a-\ol b)=k\left\langle a_c-b_c; \frac{a_w}{b_w}\right\rangle=\left\langle k(a_c-b_c); \left(\frac{a_w}{b_w}\right)^k\right\rangle$$
  $$k \ol a-k\ol b=\left\langle ka_c ;  {a_w}^k\right\rangle-\left\langle kb_c ;  {b_w}^k\right\rangle=\left\langle ka_c-kb_c; \left(\frac{a_w^k}{b_w^k}\right)\right\rangle,$$
  thus, $$k(\ol a-\ol b)=k \ol a-k\ol b.$$
  (4) For any  $\ol a\in \RI$, $k,~l\in \R$,  we have
  $$(k-l)\ol a=(k-l)\left\langle a_c; {a_w}\right\rangle=\left\langle (k-l)a_c; {a_w}^{k-l}\right\rangle,$$
  $$k \ol a-l\ol a=\left\langle ka_c; {a_w}^k\right\rangle-\left\langle la_c; {a_w}^l\right\rangle=\left\langle ka_c-la_c; \frac{a_w^k}{a_w^l}\right\rangle,$$
  thus, $$(k-l)\ol a=k \ol a-l\ol a.$$

  The proof is therefore completed.
\end{proof}

Recall that, in \cite{moo66,MH67,SB09}, the difference, the $H$-difference and the $gH$-difference of interval numbers are respectively give by
\begin{align}
\ol a\ominus \ol b&=[a_l-b_r,a_r-b_l]=\langle a_c-b_c;a_w+b_w\rangle,\label{eq4-111}\\
\ol a\ominus_H \ol b&=[a_l-b_l,a_r-b_r]=\langle a_c-b_c;a_w-b_w\rangle,\label{eq4-112}\\
\ol a\ominus_{gH} \ol b&=[\min\{a_l-b_l,a_r-b_r\},\max\{a_l-b_l,a_r-b_r\}]=\langle a_c-b_c;|a_w-b_w|\rangle.\label{eq4-113}
\end{align}

  According to  \eqref{eq4-110} -- \eqref{eq4-113}, the four differences have the same center, but their radiuses   depend on the  radius  of $\ol a$ and $\ol b$. Moreover, it is easy to prove the following statements.
\begin{proposition}  Let $\ol a,~\ol b,~\ol c,~\ol d,~\ol e\in \RI$ and let  $\ol c=\ol a-\ol b$,
   $\ol d=\ol a\ominus\ol b$, $\ol e=\ol a\ominus_{gH}\ol b$.
  \begin{itemize}
    \item[(1)] If $\frac{a_w}{b_w}\leq |a_w-b_w|$, then $\ol c\subseteq \ol e\subset \ol d;$
    \item[(2)] If $|a_w-b_w|<\frac{a_w}{b_w}< a_w+b_w$, then $\ol e\subset \ol c\subset \ol d;$
    \item[(3)] If $\frac{a_w}{b_w}\geq  a_w+b_w$, then $ \ol e\subset \ol d \subseteq \ol c;$
  \end{itemize}
\end{proposition}
\begin{example}\label{ex2}
Considering the following cases.

(1) If $\ol a=[-3,-1],~\ol b=[-4,0]$ then $ \frac{a_w}{b_w}=0.5< |a_w-b_w|=1$, and
$$	\ol c=\ol a-\ol b=[	-0.5, 0.5],\quad \ol d=\ol a\ominus\ol b=[-3,3],\quad \ol e=\ol a\ominus_{gH}\ol b=[ -1,1],$$
thus
$$\ol c\subset \ol e\subset \ol d.$$
(2) If $\ol a=[1,9],~\ol b=[-2,2]$ then $ \frac{a_w}{b_w}=2= |a_w-b_w|$, and
 $$	\ol c=[3,7],\quad \ol d=[-1,11],\quad \ol e=[3,7],$$
thus
$$\ol c= \ol e\subset \ol d.$$	
(3) If $\ol a=[-4,0],~\ol b=[-3,-1]$ then $   |a_w-b_w|=1<\frac{a_w}{b_w}=2< a_w+b_w=3$, and
 $$	\ol c=[-2,2],\quad \ol d=[-3,3],\quad \ol e=[-1,1],$$
thus
$$ \ol e\subset\ol c\subset \ol d.$$
(4) If $\ol a=[0.5,5],~\ol b=[0,1.5]$ then $    \frac{a_w}{b_w}=3= a_w+b_w$, and
 $$	\ol c=[-1,5],\quad \ol d=[-1,5],\quad \ol e=[0.5,3.5],$$
thus
$$ \ol e\subset\ol c= \ol d.$$
(5) If $\ol a=[0.5,3.5],~\ol b=[1.5,2.5]$ then $   |a_w-b_w|=1< a_w+b_w=2<\frac{a_w}{b_w}=3$, and
 $$	\ol c=[-3,3],\quad \ol d=[-2,2],\quad \ol e=[-1,1],$$
thus
$$ \ol e\subset    \ol d\subset\ol c.$$

 \end{example}
\subsubsection{Multiplication}

\begin{definition}\label{d2222}
For any $\ol a,\ol b\in \RI$,  we define
\begin{equation} \label{eq2-2}
\ol a\times\ol b=\ol a \ol b=\langle  a_c  b_c;  \ee^{\ln a_w\ln   b_w}\rangle=\left[ a_c  b_c-  \ee^{\ln a_w\ln   b_w}, a_c  b_c+  \ee^{\ln a_w\ln   b_w}\right].
\end{equation}
\end{definition}

\begin{proposition}\label{p11}
 For any  $\ol a,~\ol b,~\ol c\in \RI,$  one has
   \begin{itemize}
          \item[(1)]
   $\ol a\ol b=\ol b\ol a;$

\item[(2)] $\ol a(\ol b\ol c)=(\ol a\ol b)\ol c;$

\item[(3)]  there exists a  unique  {identity  element}  $\ol1 =\langle    1;  \ee\rangle=[1-\ee,1+\ee]\in\RI$ such that $\ol 1\ol a =\ol a;$

\item[(4)]  there exists a  unique $\ol b\in \RI$ such that $\ol a \ol b=\ol 1,$  here $\ol b $ is called the reciprocal of $\ol a$ and denoted by  $\ol b=\ol a^{-1};$

\item[(5)]    $\ol a~\ol c\leq \ol b\ol c$ if and only if  $\ol a\leq \ol b$ and $\ol c\geq \ol 0$  or $\ol a\geq \ol b$ and $\ol c<\ol  0$.
\end{itemize}
\end{proposition}
\begin{proof}(1) $\ol a\ol b=\langle  a_c  b_c; \ee^{\ln a_w\ln   b_w}\rangle=\langle  b_c a_c  ; \ee^{\ln   b_w\ln a_w}\rangle=\ol b \ol a$.

(2) $\ol a(\ol b\ol c)=\langle  a_c; a_w\rangle  \langle b_c  c_c; \ee^{\ln b_w\ln   c_w}\rangle=\langle a_cb_c c_c; \ee^{\ln a_w \ln b_w \ln c_w}\rangle.$

$(\ol a\ol b)\ol c=\langle   a_c  b_c;   \ee^{\ln a_w\ln   b_w}\rangle  \langle   c_c;   c_w\rangle=\langle a_cb_c  c_c; \ee^{\ln a_w \ln b_w \ln c_w}\rangle.$

Hence,  $\ol a(\ol b\ol c)=(\ol a\ol b)\ol c.$

(3) Let $\ol 1=\langle    1_c;  1_w \rangle,$ then we have
$$\ol a \ol 1=\langle  a_c   1_c;  a_w1_w \rangle=\langle  a_c;  a_w \rangle=\ol a.$$
This implies  $1_c=1 $ and $1_w=\ee$, i.e., $\ol1 =\langle    1;  \ee\rangle=[1-\ee,1+\ee].$
Hence, $\ol1 =[1-\ee,1+\ee]  $ is the identity element in $\RI$, and the uniqueness is obvious.

(4) Let $\ol b=\langle b_c; b_w\rangle$ such that $\ol a \ol b=\ol 1.$ Then, according to Definition \ref{d2222}, we have
$$a_cb_c=1,~~\ee^{\ln a_w\ln b_w}=\ee.$$
Thus, $$\ol b=\left\langle\frac{1}{a_c}; \ee^\frac{1}{\ln a_w}\right\rangle.$$
Hence, for any number $\ol a$, the reciprocal $\frac{~\ol 1~}{\ol a}=\ol a^{-1}=\langle\frac{1}{a_c}; \ee^\frac{1}{\ln a_w}\rangle$, and the uniqueness is obvious.

(5) (Necessity)  Since $\ol a~\ol c\leq\ol  b\ol c $, according to Definition \ref{TOrder}, we have
 $$\left\{\begin{array}{ll}
a_c c_c=b_c c_c,~\ee^{\ln a_w\ln c_w}\leq \ee^{\ln b_w\ln c_w},~\text{or}\\
a_c c_c<b_c c_c,
\end{array}\right.$$
If $\ol c\geq \ol 0$, i.e.,
$$\left\{\begin{array}{ll}
c_c=0,~c_w\geq 1,~\text{or}\\
 c_c>0,
\end{array}\right.$$
then,
$$\left\{\begin{array}{ll}
a_c =b_c ,~ a_w \leq b_w ,~\text{or}\\
a_c  <b_c ,
\end{array}\right.$$
that is,  $\ol a \leq\ol  b  $.

If $\ol c\leq \ol 0$, we can similarly prove that $\ol a \geq\ol  b  $.

(Sufficiency) If  $\ol a \leq\ol  b  $  and $\ol c\geq \ol 0$, i.e.,

$$\left\{\begin{array}{ll}
a_c =b_c ,~ a_w \leq b_w ,~\text{or}\\
a_c  <b_c ,
\end{array}\right.\qquad \text{ and }\left\{\begin{array}{ll}
c_c=0,~c_w\geq 1,~\text{or}\\
 c_c>0,
\end{array}\right.$$
then
  $$\left\{\begin{array}{ll}
a_c c_c=b_c c_c,~\ee^{\ln a_w\ln c_w}\leq \ee^{\ln b_w\ln c_w},~\text{or}\\
a_c c_c<b_c c_c,
\end{array}\right.$$
that is, $\ol a\ol c\leq\ol  b\ol c $.

The other case can be similarly proved.
The proof is therefore completed.
\end{proof}

\begin{proposition}  If  $\ol a,~\ol b,~\ol c\in \RI,$  then
   \begin{itemize}
          \item[(1)]
$(\ol a+\ol b)\ol c=\ol a~\ol c+\ol b\ol c;$

 \item[(2)] $(\ol a-\ol b)\ol c=\ol a~\ol c-\ol b\ol c;$

  \item[(3)] $k(\ol a\ol b)=(k\ol a)\ol b=\ol a(k\ol b),~ k\in \R.$
   \end{itemize}
\end{proposition}

\begin{proof}

(1) $(\ol a+\ol b)\ol c=\langle  a_c+b_c ; a_wb_w\rangle   \langle  c_c;    c_w\rangle=\langle (a_c+b_c) c_c; \ee^{\ln (a_w  b_w) \ln c_w}\rangle.$

$\ol a~\ol c+\ol b\ol c=\langle   a_c  c_c;   \ee^{\ln a_w\ln  c_w}\rangle+  \langle  b_c c_c;  \ee^{\ln b_w\ln   c_w}\rangle=\langle a_cc_c+b_c  c_c; \ee^{\ln a_w\ln c_w+ \ln b_w \ln c_w}\rangle.$

Hence,  $(\ol a+\ol b)\ol c=\ol a~ \ol c+\ol b\ol c.$

(2) The proof is similar to that of (1).

(3)  $k(\ol a\ol b)=k\langle   a_c  b_c;   \ee^{\ln a_w\ln  b_w}\rangle=\left\langle   ka_c  b_c;   \ee^{k\ln a_w\ln  c_w}\right\rangle$;

$(k\ol a)\ol b= \langle  k a_c  ;     a_w^k\rangle\langle  b_c  ;     b_w\rangle=\left\langle   ka_c  b_c;   \ee^{\ln a_w^k\ln  c_w}\right\rangle;$

$\ol a(k\ol b)=\langle   a_c  ;     a_w^k\rangle\langle  kb_c  ;     b_w^k\rangle=\left\langle   ka_c  b_c;   \ee^{\ln a_w\ln  c_w^k}\right\rangle.$

Hence, $k(\ol a\ol b)=(k\ol a)\ol b=\ol a(k\ol b).$

 The proof is therefore completed.
\end{proof}

Recall that the multiplication defined in \cite{moo66} is given as
\begin{align}\label{c-malt}\ol d=\ol a\otimes\ol b=[\min\{a_lb_l,a_lb_r,a_rb_l,a_rb_r\},
\max\{a_lb_l,a_lb_r,a_rb_l,a_rb_r\}].
\end{align}

Form   \eqref{eq2-2} and \eqref{c-malt},
it is easy to see that $\ol a\times\ol b$ and $\ol a\otimes\ol b$ are uncomparable. For example, considering the   cases in Table \ref{table1},

\begin{table}[h]
  \centering
  \begin{tabular}{|c|c|c|c|}\hline
  $\ol a$& $\ol b$& $\ol c=\ol a\ol b$& $\ol d=\ol a\otimes\ol b$\\\hline

$[-10, 5]$&$[ -10, 5]$& $[-51.71, 64.21]$&$[-50, 100]$\\\hline
$[-2, -1]$&[1, 2]&$[-3.87, -0.63]$&$[-4, -1]$\\\hline
[0, 1] &[0, 1] &$[-1.37, 1.87]$&[0, 1]\\\hline
[0, 1]&[0, 2]&$[-0.5, 1.5]$&[0, 2]\\\hline
[0, 1]& [0, 4] & [0.38, 1.62]&[0, 4]\\\hline
  \end{tabular}
  \caption{Different cases for the new multiplication \eqref{eq2-2} and the classical one \eqref{c-malt}.}\label{table1}
\end{table}

Now we give the definition of the $n$th power of interval numbers.
\begin{definition}Let $\ol a\in  \RI$,    we then define the $n$th power of $\ol a$ by
\begin{equation} \label{eq2-3}
\ol a^n=\ol a\times \ol a\times...\times \ol a=\left\langle    a_c^n;   \ee^{\ln^n a_w }\right\rangle,~~~  n=0,1,2,....
\end{equation}
Especially, $\ol a^0=\ol 1$.
\end{definition}

\begin{proposition}  If  $\ol a,~\ol b\in \RI,$ $k\in \R$, $n,l\in \N_+$  then
   \begin{itemize}
     \item[(1)]
   $ \ol a ^n \ol a ^l=\ol a^{n+l};$

\item[(2)]   $(k\ol a)^n=k^n\ol a^n;$

\item[(3)] $(\ol a\ol b)^n=\ol a^n \ol b^n;$

\item[(4)] $(\ol a+\ol b)^n=\sum\limits_{i=1}^{n} C_n^i~  \ol a^i~ \ol b^{n-i}.$
    \end{itemize}
\end{proposition}

\begin{proof}    (1)-(3) are easy to prove, so we only prove (4).

For (4), we have  \begin{align*}
  (\ol a+\ol b)^n=&\langle a_c+b_c; a_w b_w\rangle^n\\
  =&\left\langle (a_c+b_c)^n;\ee^{\ln^n (a_w b_w)}\right\rangle\\
  =&\left\langle \sum\limits_{i=1}^{n} C_n^i~    a_c^i~   b_c^{n-i};\ee^{\sum\limits_{i=1}^{n} C_n^i~   \ln ^i a_w~  \ln^{n-i} b_w}\right\rangle\\
  =&\sum\limits_{i=1}^{n} C_n^i~   \left\langle  a_c^i~   b_c^{n-i};\ee^{   \ln ^i a_w~  \ln^{n-i} b_w}\right\rangle\\
    =&\sum\limits_{i=1}^{n} C_n^i~  \ol a^i~ \ol b^{n-i}.
\end{align*}

The proof is therefore completed.
\end{proof}

\subsubsection{Division}

\begin{definition}\label{Division}
Let $\ol a,\ol b\in \RI$,  if $b_c\neq0$ and $b_w\neq 1 $,  we define
\begin{equation} \label{eq2-4}
\ol a\div\ol b=\frac{~\ol a~}{ \ol b}=\left\langle  \frac{~a_c~}{  b_c};  \ee^{\frac{~\ln a_w~}{\ln   b_w}}\right\rangle=\left[  \frac{~a_c~}{  b_c}-  \ee^{\frac{~\ln a_w~}{\ln   b_w}},\frac{~a_c~}{  b_c}+ \ee^{\frac{~\ln a_w~}{\ln   b_w}}\right].
\end{equation}
\end{definition}

\begin{lemma}\label{p112}
 Let $\ol a,\ol b\in \RI$,  if $b_c\neq0$ and $b_w\neq 1 $, then
\begin{equation} \label{eq221}
\ol a\div\ol b=\ol a { \ol b}^{-1} =\ol a\frac{~\ol 1~}{ \ol b}.
\end{equation}
\end{lemma}
\begin{proof}
  By Definition  \ref{d2222} and Proposition \ref{p11}, we have
  $$\ol a { \ol b}^{-1} =\left\langle {a_c};   a_w \right\rangle\left\langle\frac{1}{b_c}; \ee^\frac{1}{\ln b_w}\right\rangle=\left\langle  \frac{~a_c~}{  b_c};  \ee^{\frac{~\ln a_w~}{\ln   b_w}}\right\rangle.$$
  Thus according to Definition  \ref{Division},  \eqref{eq221} holds.
\end{proof}

Recall that the multiplication defined in \cite{moo66} is given as
\begin{align}\label{c-div}
\ol d=\ol a\oslash\ol b=\left[\min\left\{\frac{a_l}{b_l},\frac{a_l}{b_r},\frac{a_r}{b_l},\frac{a_r}{b_r}\right\},\max\left\{\frac{a_l}{b_l},\frac{a_l}{b_r},\frac{a_r}{b_l},\frac{a_r}{b_r}\right\}\right],
\end{align}
where $0\not\in \ol b$.

According to  \eqref{eq2-4} and \eqref{c-div}, it is easy to see that the divisions  $\ol a\div\ol b$ and $\ol a\oslash\ol b$ are  uncomparable. For example, considering the   cases in Table \ref{table2}, where ``$\setminus$" implies that the division does not exist.

\begin{table}[h]
  \centering
  \begin{tabular}{|c|c|c|c|}\hline
  $\ol a$& $\ol b$& $\ol c=\ol a\div b$& $\ol d=\ol a\oslash\ol b$\\\hline
$[-7,-5]$&$[-10,-2]$&$[0,2]$&$[0.5,3.5]$\\\hline
$[2,4]$&$[1,5]$&$[0,2]$&$[0.4,4]$\\\hline
$[7,9]$&$[-10,-6]$&$[-2,0]$&$[-1.5,-0.7]$\\\hline
$[-6,-4]$&$[-10,0]$&$[0,2]$&$\setminus$\\\hline
$[-10,0]$&$[-6,-4]$&$\setminus$&$[0,2.5]$\\\hline
$[-1,1]$&$[-1,2]$&$[-1,1]$&$\setminus$\\\hline
$[-1,2]$&$[-2,2]$&$\setminus$&$\setminus$\\\hline
  \end{tabular}
  \caption{Different cases for the new  division \eqref{eq2-4} and the classical one \eqref{c-div}. }\label{table2}
\end{table}

\begin{proposition}\label{p111}
 Let $\ol a,~\ol b ,~\ol c\in \RI$ with  $\ol a\ol b=\ol c$, if    $b_c\neq0$ and $b_w\neq 1 $, then
\begin{equation} \label{eq2212}
\ol a=\frac{~\ol c~}{\ol b}.
\end{equation}
\end{proposition}
\begin{proof}
 Since $\ol a\ol b=\ol c$, we have
  $$\left\langle {c_c};   c_w \right\rangle =\left\langle {a_c}{b_c};   \ee^{\ln a_w\ln b_w} \right\rangle.$$
  Thus, according to Definition  \ref{Division}
  $$\frac{~\ol c~}{\ol b}=\left\langle  \frac{~c_c~}{  b_c};  \ee^{\frac{~\ln c_w~}{\ln   b_w}}\right\rangle=\left\langle {a_c};   a_w \right\rangle=\ol a.$$
   This ends the proof.
\end{proof}

\begin{remark} From Proposition \ref{p111}, it is easy to see that the division is the inverse of the multiplication, while the classical one is not.

\end{remark}
\subsection{Arithmetic Operations Between Interval Numbers and Real Numbers}

As we mentioned before, in this paper, we do not consider   real numbers as   degenerate interval numbers, so we need to define the arithmetic operations between interval numbers and real numbers.

 First, we define a bijection from real numbers to interval numbers.
\begin{definition}  For any $\lambda\in\R$, we define
$$\ol \lambda:=\lambda\ol 1=\langle \lambda ;\ee^\lambda \rangle=[\lambda -\ee^\lambda,\lambda +\ee^\lambda].$$
\end{definition}

\begin{definition}\label{def3222}  For any $\lambda\in\R$, $\ol a\in\RI$, we define
$$\ol a\star\lambda:=\ol a\star\lambda\ol 1=\ol a\star\ol \lambda,$$$$   \lambda\star\ol a:=\lambda\ol 1\star \ol a=\ol \lambda\star\ol a,$$
where $\star\in \{+,-,\times,\div\}$.
\end{definition}

\begin{remark} In Definition \ref{def3222}, if $\star=\times$, then it is the same as the scale multiplication given in Definition \ref{d21}.
\end{remark}

\begin{example} $\ol a=[-1,3]$, $\lambda =2$, then
\begin{align*}
  \ol a+\lambda&=\left[3-2\ee^2,3+2\ee^2\right]=\lambda+\ol a,\\
  \ol a-\lambda&=\left[-1-2\ee^{-2},-1+ 2\ee^{-2}\right],\\
  \lambda-\ol a&=\left[1-\frac12\ee^{2},1+ \frac12\ee^{2}\right],\\
  \ol a\times\lambda&=[-2,6]=\lambda \ol a,\\
  \ol a\div\lambda&=\left[\frac12-\sqrt{2},\frac12 +\sqrt{2}\right],\\
  \lambda\div\ol a&=\left[2-\ee^{\frac{2}{\ln 2}},2+\ee^{\frac{2}{\ln 2}}\right],
\end{align*}
but
  \begin{align*}
  \ol a\oplus\lambda&=\left[1 ,5 \right]=\lambda\oplus\ol a,\\
  \ol a\ominus\lambda&=\left[-3,1\right]=\ol a\ominus_{gH}\lambda=\ol a\ominus_{H}\lambda,\\
  \lambda\ominus\ol a&=\left[-1,3\right]=\lambda\ominus_{gH}\ol a,\\
  \ol a\otimes\lambda&=[-2,6]=\lambda\otimes\ol a,\\
  \ol a\oslash\lambda&=\left[-\frac12 ,\frac32 \right],\\
  \lambda\oslash\ol a & \text{~~does not exist}.
\end{align*}

\end{example}

\subsection{Distance, Norm, and Inner Product}

This subsection will present new distance, norm, and inner product on $\RI$, and   show that $\RI$  is a Hibert space.

We first show that the space $\RI$ with the operation ``$\cdot$" and ``$+$" defined in Definitions \ref{d2} and \ref{d21}, respectively, is a linear space.

\begin{theorem}
  $\RI$ is a linear space.
\end{theorem}
\begin{proof}
According to Propositions \ref{p1} and \ref{p2}, it is clear that $\RI$ is a linear space.
\end{proof}

\subsubsection{Distance}

\begin{definition}\label{dist} For any $\ol a,\ol b\in \RI$,  we define
   $$d(\ol a,\ol b)=\sqrt{  (a_c-b_c)^2+(\ln a_w-\ln b_w)^2}.$$
\end{definition}

By Definition \ref{dist}, it is easy to see that
\begin{proposition}\label{p32}
 Assume that $\ol a,\ol b,\ol c\in \RI$,  then we have
   \begin{itemize}
     \item[(1)]  $d(\ol a,\ol b)\geq 0$, and  { $d(\ol a,\ol b)= 0$ \text{implies} $\ol a=\ol b$};
     \item[(2)] $d(k\ol a,k\ol b)=|k| d(\ol a,\ol b)$, $k\in\R$;
     \item[(3)]  $d(\ol a,\ol b)\leq d(\ol a,\ol c)+d(\ol c,\ol b)$.
   \end{itemize}
\end{proposition}
   Thus $d(.,.)$ is a distance on $\RI$.
 Hence, $\RI$ is a  distance space.

\subsubsection{Norm}

\begin{definition}\label{norm} For any $\ol a\in \RI$,  we define
   $$\|\ol  a\|= d(\ol a,\ol 0)=\sqrt{  a_c^2+\ln^2 a_w}.$$
\end{definition}
By Definition \ref{norm}, it is easy to see that
\begin{proposition}\label{p3}
 Assume that $\ol a,\ol b\in \RI$,   then we have
   \begin{itemize}
     \item[(1)]  $\|\ol a\|\geq 0$, and  { $\|\ol a\|= 0$ \text{implies} $\ol a=\ol 0$};
     \item[(2)] $\| k\ol a\|=|k| \| \ol  a\|$, $k\in\R$;
     \item[(3)]  $\|  \ol a+\ol b\|\leq \| \ol  a\|+\|\ol  b\|$.
   \end{itemize}
\end{proposition}
   Thus $\|\cdot\|$ is a norm on $\RI$.

\begin{theorem}\label{thm3.2}
The space $\RI$ endowed with the norm $\|\cdot\|$ is a normed space.
\end{theorem}
\begin{proof}
According to Proposition \ref{p3}, the theorem can be easily  proved.
\end{proof}

\begin{lemma}\label{law}
  For any $a,b\in \RI$,
   $$2\|\ol a\|+2\|\ol b\|=   \|\ol a+\ol b\|^2+\|\ol a-\ol b\|^2 .$$
\end{lemma}
\begin{proof} By Definition \ref{norm}, we have
  $$\|\ol a+\ol b\|^2=(a_c+b_c)^2+(\ln a_w+\ln b_w)^2,$$
  and
  $$\|\ol a-\ol b\|^2=(a_c-b_c)^2+(\ln a_w-\ln b_w)^2.$$
  Thus,
  $$ \|\ol a+\ol b\|^2+\|\ol a-\ol b\|^2=2a_c^2+2b_c^2+2\ln^2 a_w+2\ln^2 b_w=2\|\ol a\|+2\|\ol b\|.$$
\end{proof}

Based on the norm, we give the  definition of  the limit in $\RI$.

\begin{definition} Let $\ol a\in \RI$ and  $\{\ol a_n\}_{n=1}^{\infty}\subset \RI$, we say $\ol a$ is the limit of the sequence $\{\ol a_n\}_{n=1}^{\infty}$,
if
$$\|\ol a_n-\ol a\|\to 0, \text{ ~ as ~} n\to \infty,$$
   and denoted by $$\ol a=\lim_{n\to\infty} \ol a_n.$$
\end{definition}

\begin{theorem}\label{thmlim}
 Let $\ol a=\left\langle a_c;a_w\right\rangle$ and   $\ol a_n=\left\langle(a_n)_c;(a_n)_w\right\rangle$, $n=1,2,....$ Then  $\ol a=\lim_{n\to\infty} \ol a_n $ if and only if
 $$  a_c=\lim_{n\to\infty}   (a_n)_c, \text{~~~~ and~~~~} a_w=\lim_{n\to\infty}   (a_n)_w.$$
\end{theorem}

\begin{proof}
(Necessity)  Since $\ol a=\lim_{n\to\infty} \ol a_n $, one has
$$  \|\ol a_n-\ol a\|= \sqrt{  ((a_n)_c-a_c)^2+(\ln (a_n)_w-\ln a_w)^2}\to 0,  \text{ as } n\to\infty,$$
which implies that
$$(a_n)_c\to a_c, \qquad \ln (a_n)_w\to\ln a_w,  \text{ as } n\to\infty,$$
i.e.,
 $$  a_c=\lim_{n\to\infty}   (a_n)_c, \text{~~~~ and~~~~} a_w=\lim_{n\to\infty}   (a_n)_w.$$

(Sufficiency) The sufficiency is obvious according to the necessity.
\end{proof}

\begin{theorem}\label{thm3.3}
$\RI$ is a Banach space endowed with the  norm $\|\cdot\| $.
\end{theorem}

\begin{proof} Let $\{\ol a_m|\ol a_m=\langle (a_c)_m;(a_w)_m\rangle\},$ $m=1,2,3,...,$ be a Cauchy sequence in  $\RI$.
Thus, for all $\epsilon>0$, there exists $N\in \N_+$, when $m,n>N$, one has
$$  \|\ol a_m-\ol a_n\|= \sqrt{  ((a_m)_c-(a_n)_c)^2+(\ln (a_m)_w-\ln (a_n)_w)^2}<\epsilon,$$
which implies $|(a_m)_c-(a_n)_c|<\epsilon$ and  $|\ln (a_m)_w-\ln (a_n)_w|<\epsilon$. Thus,
 $\{  (a_m)_c\} $ and $\{  \ln (a_m)_w\} $ are also    Cauchy sequences in  $ \R $.
Therefore, there exist $  b_c,~b_w\in\R$ such that
 $$\lim_{m\to\infty}(a_m)_c=b_c, \qquad \lim_{m\to\infty}\ln (a_m)_w=b_w. $$
 Hence, for arbitrary  $\varepsilon>0$, there exists $N>0$ such that
 $$|(a_m)_c-b_c |< \varepsilon/\sqrt{2}  , \qquad |\ln (a_m)_w-b_w|< \varepsilon/\sqrt{2} ,$$
when $m>N$.
Therefore, for $m>N$,  we have
$$\|\ol a_m-\ol b \|\leq \sqrt{  ((a_m)_c-b_c)^2+(\ln (a_m)_w-b_w)^2}  < \varepsilon,$$
where $\ol b=[b_c-\ee^{b_w},b_c+\ee^{b_w}]$. This implies that $\{\ol a_m\}$ is convergent to $\ol b  $ in $\RI$. Therefore, the proof is completed.
\end{proof}

\subsubsection{Inner Product}

From Lemma \ref{law}, we see that the norm   in Definition \ref{norm} satisfies the parallelogram law,  so the inner product on $\RI$ can be given as follows.
\begin{definition}\label{IPd} For any $a,b\in \RI$,  we define
   $$\ol a\diamond\ol b = \frac14(\|\ol a+\ol b\|^2-\|\ol a-\ol b\|^2)= a_c  b_c+\ln a_w\ln b_w.$$
\end{definition}

\begin{proposition} Assume that $\ol a,~\ol b,~\ol c\in \RI$, $k,~l\in\R$, then we have
   \begin{itemize}
     \item[(1)] {Commutativity:} $\ol a\diamond\ol  b=\ol b\diamond\ol  a$;
     \item[(2)]{ Linearity:} $(k\ol a+l\ol b)\diamond \ol c=k \ol  a \diamond\ol c+l  \ol b\diamond\ol c$;
     \item[(3)] {Positive-Definiteness:}  $ \ol a\diamond \ol a \geq 0$, and { $  \ol a \diamond \ol a = 0 \Leftrightarrow \ol a= \ol 0$}.
   \end{itemize}
\end{proposition}
\begin{proof}
(1) It is clear that $$  \ol a\diamond\ol  b =a_c  b_c+\ln a_w\ln b_w=b_c  a_c+\ln b_w\ln a_w=\ol b\diamond\ol a.$$

(2) It follows that
\begin{align*}
  ( k\ol a+l\ol b)\diamond\ol  c&=\langle ka_c+l b_c; {a_c}^k {b_c}^l \rangle\diamond\langle c_c;c_w \rangle\\
  &=   (ka_c+l b_c)c_c+ \ln({a_c}^k {b_c}^l)\ln c_w\\
  &=  ka_cc_c+  k\ln {a_c} \ln c_w+l b_cc_c +  l\ln {b_c} \ln c_w\\
   &= k\ol a \diamond\ol  c  +l\ol b\diamond\ol  c.
\end{align*}

(3) We observe that $\ol a\diamond \ol a= a_c^2+  \ln^2a_w\geq 0$ and { $\ol a\diamond \ol a= 0 \Leftrightarrow \ol a= \ol 0$}.
\end{proof}

Thus, ``$\diamond$" is an inner product on $\RI$.

\begin{lemma}\label{lem3.1}  $\RI$ is an inner product space under the  inner product ``$\diamond $".
\end{lemma}

\begin{theorem}\label{thm3.5} $\RI$ is a Hilbert space.
\end{theorem}
\begin{proof} By Definition \ref{IPd}, we have  $\|\ol a\|=\sqrt{\ol a\diamond \ol a }$ for any $\ol a\in\RI$. This implies that the norm is induced from the inner product. It follows from Theorem \ref{thm3.3} that $\RI$ is a Banach space endowed with the  above norm $\|\cdot\| $. Thus, we complete the proof.
\end{proof}

\section{The interval-valued function}
Let $[a,b]\subset \R$,
a mapping $\ol f:[a,b]\to \RI$ is called an interval-valued function (IVF).
For any  IVF $\ol f:[a,b]\to \RI$, there exist two real functions $f_l,f_r:[a,b]\to \R$ such that
$$\ol f(x)=[f_l(x),f_r(x)], \qquad  x\in [a,b],$$
where $f_l,f_r$ are called the left and right endpoint functions of $\ol f$, respectively.
Now let
$$f_c(x)=\frac{f_l(x)+f_r(x)}{2},  \qquad  f_w(x)=\frac{f_r(x)-f_l(x)}{2},  \qquad  x\in [a,b],$$
then $f_c$ is called the \textit{center function} of $\ol f$, and  $f_w$ is called the \textit{radius function} of $\ol f$.

\begin{definition}
   Let $\ol f, \ol g: [a,b]\to\RI$, we define
$$\ol f(x)\star\ol g(x):= \langle f_c(x);f_w(x)\rangle\star \langle g_c(x);g_w(x)\rangle,\qquad x\in [a,b],$$
where $\star\in \{+,-,\times,\div\}$.
\end{definition}

\begin{example}\label{ex3}
Considering that $\ol f(x)=[0,2]-[-1,1]|x-2| $  and $\ol g(x)=[0,2]\ominus_{gH}[-1,1]\odot|x-2|$, $x\in[0,4]$.
Thus,  $\ol f(x)=[0,2]$ on [0,4] and
\begin{align*}
  \ol g(x)=\begin{cases}
[|x-2|,2-|x-2|],& x\in[0,1)\cup(3,4],\\
[2-|x-2|,|x-2|],& x\in[1,3].
\end{cases}
\end{align*}
 see Figure \ref{f4}.

\begin{figure}[h]
  \centering
  \includegraphics[height=6cm]{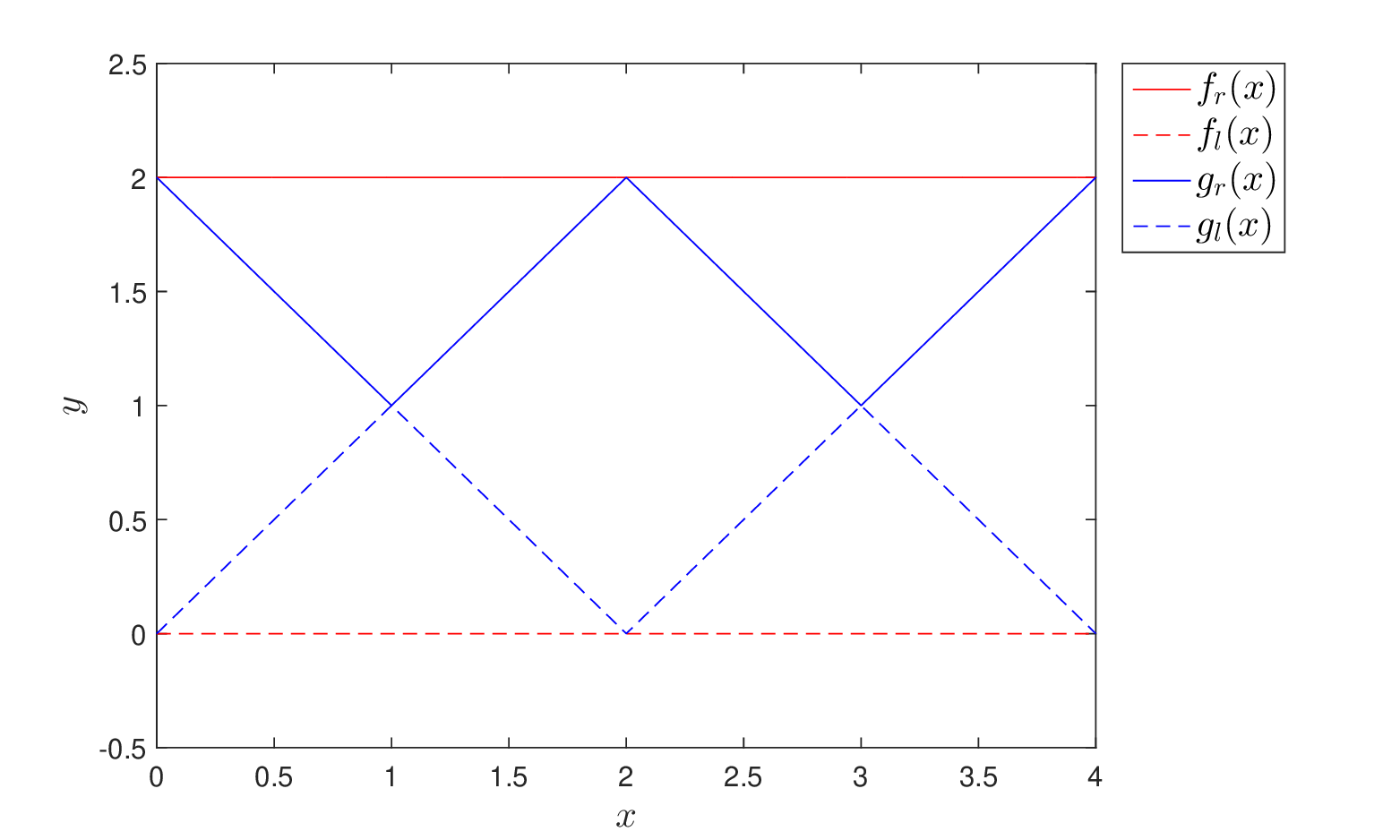}
  \caption{ IVFs $\ol f $ and $\ol g$ in Example \ref{e3}}\label{f4}
\end{figure}
\end{example}

\begin{definition}\label{def32222}
Let $\ol f: [a,b]\to\RI$, $    g: [a,b]\to\R$, we define
$$\ol f(x)\star  g(x):= \ol f(x)\star  (\ol 1~g(x)),\qquad x\in [a,b],$$
$$g(x)\star  \ol f(x):=  (\ol 1~g(x)) \star \ol f(x) ,\qquad x\in [a,b],$$
where $\star\in \{+,-,\times,\div\}$.
\end{definition}

Now we introduce the continuity of IVFs.
\begin{definition}\label{CIdef}
 Let $\ol f:[a,b]\to \RI$ be   an interval-valued function, given a point $x\in[a,b]$, if for arbitrary $\varepsilon>0$, there exists $\delta>0$ such that
$$d(\ol f(x),\ol f(y))<\varepsilon,$$
when $y\in O(x,\delta)\cap [a,b]$. Then $\ol f$ is called continuous at $x$.  Further, $\ol f$ is called continuous on $[a,b]$ if $\ol f$ is continuous on every $x\in[a,b]$.
\end{definition}

Denoted by $C([a,b],\RI) $ the space of all continuous IVFs on $[a,b]$.

\begin{theorem}\label{thmcont}
Let $\ol f:[a,b]\to \RI$ be   an interval-valued function. Then $\ol f$ is continuous on $[a,b]$ if and only if   $f_l$ and $f_r$  are continuous on $[a,b]$.
\end{theorem}

\begin{proof}  (Sufficiency)  Since  $f_l$ and $f_r$  are continuous on $[a,b]$, then
$f_c=\frac{f_l+f_r}{2}$ and $f_w=\frac{f_r-f_l}{2}$ are also continuous on $[a,b]$, so is $\ln f_w$.
Let $x\in[a,b]$, then for arbitrary $\varepsilon_1>0$, there exists $\delta_1>0$ such that
$$|f_c(x)-f_c(y)|\leq \varepsilon_1,$$
where $y\in O(x,\delta_1)\cap [a,b]$.
Also, for arbitrary $\varepsilon_2>0$, there exists $\delta_2>0$ such that
$$|\ln f_w(x)-\ln f_w(y)|\leq\varepsilon_2,$$
where $y\in O(x,\delta_2)\cap [a,b]$.
Let $\delta=\min\{\delta_1,\delta_2\}$,  $\varepsilon=\max\{\varepsilon_1,\varepsilon_2\}$,
Thus,
$$d(\ol f(x),\ol f(y))=\sqrt{|f_c(x)-f_c(y)|^2+|\ln f_w(x)-\ln f_w(y)|^2}\leq \sqrt{\varepsilon_1^2+\varepsilon_2^2}\leq \sqrt{2} \varepsilon,$$
where $y\in O(x,\delta)\cap [a,b]$. By Definition \ref{CIdef}, $\ol f$ is continuous on $[a,b]$.

  (Necessity)  Since $\ol f$ is continuous on $[a,b]$,  for arbitrary $\varepsilon>0$, there exists $\delta>0$ such that
$$d(\ol f(x),\ol f(y))=\sqrt{|f_c(x)-f_c(y)|^2+|\ln f_w(x)-\ln f_w(y)|^2}\leq \varepsilon,~~ x\in[a,b],$$
where  $y\in O(x,\delta)\cap [a,b]$. Thus,
 $|f_c(x)-f_c(y)|\leq \varepsilon$, and $|\ln f_w(x)-\ln f_w(y)|\leq \varepsilon$,
  which imply $f_c$ and $\ln f_w$ are also continuous on $[a,b]$, so is $  f_w$.
  Therefore,
  $f_l= {f_c-f_w}$ and $f_r= {f_c+f_w} $ are also continuous on $[a,b]$.
\end{proof}

\begin{definition}\label{cnorm}
  For any $\ol f\in C([a,b],\RI) $, we define
  $$\|\ol f\|_{\infty}=\max_{x\in[a,b]}\|\ol f(x)\|.$$
\end{definition}

By Definition \ref{cnorm}, it is easy to see that
\begin{proposition}\label{cp3}
 Assume that $\ol f,\ol g\in C([a,b],\RI)$, $k\in\R$, then we have
   \begin{itemize}
     \item[(1)]  $\|\ol f\|_{\infty}\geq 0$, and  { $\|\ol f\|_{\infty}= 0$ if and only if $\ol f=\ol 0$};
     \item[(2)] $\| k\ol f\|_{\infty}=|k| \| \ol  f\|_{\infty}$;
     \item[(3)]  $\|  \ol f+\ol g\|_{\infty}\leq \| \ol  f\|_{\infty}+\|\ol  g\|_{\infty}$.
   \end{itemize}
\end{proposition}
\begin{proof}(1)
$\|\ol f\|_{\infty} =\max_{x\in[a,b]}\|\ol f(x)\|=\max_{x\in[a,b]}\sqrt{f_c^2(x)+\ln^2 f_w(x)}\geq 0.$
Thus, { $\|\ol f\|_{\infty}= 0$ if and only if $\ol f=\ol 0$}.

(2) $\| k\ol f\|_{\infty}=\max_{x\in[a,b]}\|k\ol f(x)\|=\max_{x\in[a,b]}\sqrt{k^2f_c^2(x)+k^2\ln^2 f_w(x)} =|k| \| \ol  f\|_{\infty}$.

(3) \begin{align*}
  \|  \ol f+\ol g\|_{\infty}&=\max_{x\in[a,b]}\sqrt{(f_c(x)+g_c(x))^2+(\ln   f_w(x)+\ln g_w(x))^2}\\
  &\leq \max_{x\in[a,b]}\sqrt{\| \ol  f\|_{\infty}^2+2f_c(x)g_c(x)+2\ln   f_w(x)\ln    g_w(x)   +\| \ol  g\|_{\infty}^2}\\
  &\leq\sqrt{\| \ol  f\|_{\infty}^2+2 \| \ol  f\|_{\infty}\| \ol  g\|_{\infty}   +\| \ol  g\|_{\infty}^2}\\
  &=\| \ol  f\|_{\infty}+\|\ol  g\|_{\infty}.
\end{align*}
This ends the proof.
\end{proof}
According to Proposition \ref{cp3}, $\|\cdot\|_{\infty}$ is a norm on $C([a,b],\RI)$, and
\begin{lemma}\label{thm3.21}
The space $C([a,b],\RI)$ endowed with the norm $\|\cdot\|_{\infty}$ is a normed space.
\end{lemma}

Further, we can prove that
 \begin{theorem}\label{thm3.22}
The space $C([a,b],\RI)$ endowed with the norm $\|\cdot\|_{\infty}$ is a Banach space.
\end{theorem}
\begin{proof}
  The proof is similar to that of Theorem \ref{thm3.3}, so we omit it.
\end{proof}

\section{Derivative of IVFs}

In this section, we will present a new Derivative of IVFs based on the new subtraction given in Definition \ref{subt}, and
discuss  its relationship with the classical derivative,  the multiplicative derivative given in \cite{BKO08}, and the $gH$-derivative given in \cite{SB09}, respectively.

\subsection{Basic Definition and Properties}

\begin{definition}\label{def1}
Suppose that the IVF $\ol f:[a,b]\rightarrow \RI$ and $x_0\in [a,b]$ and $x_0+h\in[a,b]$, then $\ol f$ is said to be differentiable at $x_0$ if there exists $\ol f'(x_0)\in \RI$ such that
\begin{equation}
\ol f'(x_0)=\lim\limits_{h\rightarrow 0}\frac{\ol f(x_0+h)-\ol f(x_0)}{h}.
\end{equation}
We say $\ol f$ is  differentiable on $[a,b]$ if it is differentiable at each $x_0\in[a,b]$.
\end{definition}

\begin{theorem}\label{thmdiff}
  Let $\ol f:[a,b]\to \RI$ be an IVF  with $\overline f(x)=\left\langle  {f}_c(x); {f}_w(x) \right\rangle $. Then $\overline f$  is differentiable at  $x\in[a,b]$ if and only if $ {f}_c$ is differentiable at $x$, and  ${f}_w$ is multiplicative differential at $x$. Moreover,
  $$\overline f'( x )=\left\langle {f}_c'(x);{f}_w^*(x) \right\rangle,$$
  where ${f}_w^*$ is the multiplicative derivative of $f_w$ on $[a,b]$ (see details in \cite{BKO08}), i.e.,
  \begin{equation}\label{mulder}
  f^*(x)=\lim_{h\to 0}\left(\frac{{f}  (x+h)}{{f} (x)}\right)^{\frac1h}=\ee^{(\ln f)'(x)}.
\end{equation}
\end{theorem}

\begin{proof}By Definition \ref{def1},
\begin{eqnarray*}
\lim_{h\to 0}\frac{\overline f\left( {{x} }+h \right)-\overline f\left( {{x} } \right)}{h}
&=&\lim_{h\to 0}\frac{ \langle  f_c( {{x} }+h  ); f_w ( {{x} }+h  )\rangle-\langle f_c( {{x} }   ); f_w ( {{x} }  )\rangle}{h}\\
&=&\lim_{h\to 0} \left\langle \frac{f_c( {{x} }+h  )-f_c( {{x} }   )}{h}; \left(\frac{f_w ( x +h  )}{ f_w ( {{x} })  }\right)^{\frac1h}\right\rangle \\
&=& \left\langle {f}_c'(x_0);{f}_w^*(x_0) \right\rangle.
\end{eqnarray*}
This ends the proof.
\end{proof}

\begin{remark} Theorem \ref{thmdiff} implies that the derivative of an IVF $\ol f$ on $[a,b]$  equals to  the classical  derivative of the center function $f_c$, and the multiplicative derivative of the radius function $f_w$.
\end{remark}

\begin{definition}\label{primitive}
Let $\ol F, \ol f  :[a,b]\to \RI$, if
$$
\ol F'(x)=  \ol f(x) , ~~   x \in [a, b],
$$
then
$\ol F  $ is called a primitive of $\ol f$  on $[a,b]$.
\end{definition}

\begin{proposition}\label{prop4.5} Suppose that $\overline f,\overline g:[a,b]\to \RI$ are differentiable on $[a,b]$. Then for any $x\in[a,b]$,
\begin{itemize}
\item[(1)] ${{\left(c_1 \overline f+ c_2\overline g \right)}{'}}\left( {{x}} \right)=c_1{{\overline f}{'}}\left( {{x}} \right)+c_2{{\overline g} {'}}\left( {{x} } \right),$ where  $c_1,c_2\in\RI.$

\item[(2)] $(\overline f~\overline g )'(x)=\overline f'(x)\overline g(x)+\overline f(x)\overline g'(x).$

\item[(3)] $\left(\frac{~ \overline f ~}{\overline g}\right)'(x)=\displaystyle\frac{\overline f'(x)\overline g(x)-\overline f(x)\overline g'(x)}{\overline g^2(x)}.$
\end{itemize}

\end{proposition}

\begin{proof}(1)
\begin{eqnarray*}
{{\left(c_1 \overline f+ c_2\overline g \right)}{'}}\left( {{x} } \right)
&=&\lim_{h\to 0}\frac{(c_1 \overline f(x_0+h)+ c_2\overline g(x_0+h))-(c_1 \overline f(x_0)+ c_2\overline g(x_0))}{h}\\
&=&\lim_{h\to 0}\left(c_1 \frac{\overline f(x_0+h)-  \overline f(x_0) }{h}+c_2\frac{ \overline g(x_0+h)-  \overline g(x_0) }{h}\right)\\
&=&c_1{{\overline f}{'}}\left( {{x} } \right)+c_2{{\overline g} {'}}\left( {{x} } \right).
\end{eqnarray*}

(2)
\begin{eqnarray*}
{{\left(  \overline f(x) \overline g \right)}{'}}\left( {{x} } \right)
&=& \langle f_c(x)g_c(x);\ee^{\ln f_w(x)\ln  g_w(x)} \rangle'\\
&=&\langle f_c'(x)g_c(x)+ f_c(x_0)g_c'(x);\ee^{(\ln f_w)'(x)\ln g_w(x)+\ln f_w(x)(\ln g_w)'(x)} \rangle  \\
&=&\langle f_c'(x)g_c(x) ;\ee^{\ln f_w'(x)\ln g_w(x) } \rangle +\langle   f_c(x)g_c'(x); \ee^{\ln f_w(x)(\ln g_w)(x)} \rangle  \\
&=&\langle f_c'(x) ;f_w^*(x)  \rangle\langle g_c(x);  g_w(x)   \rangle +\langle   f_c(x);f_w(x)  \rangle\langle   g_c'(x);  g_w^*(x) \rangle
  \\
 &=&\overline f'(x)\overline g(x)+\overline f(x)\overline g'(x).
\end{eqnarray*}

(3) \begin{align*}
  \left(\frac{~\overline f~}{\overline g}\right)'(x)
  =&\lim_{h\to 0}\frac{\frac{\ol f(x+h)}{\ol g(x+h)}-\frac{\ol f(x)}{\ol g(x)}}{h}\\
  =&\lim_{h\to 0}\left\langle\frac{\frac{ f_c(x+h)}{ g_c(x+h)}-\frac{ f_c(x)}{ g_c(x)}}{h};\ee^{\frac{1}{h}\left(\frac{\ln f_w(x+h)}{\ln g_w(x+h)}-\frac{\ln f_w(x)}{\ln g_w(x)}\right)}\right\rangle\\
  =& \left\langle\frac{ { f_c'(x)}{ g_c(x)}- { f_c(x)}{ g_c'(x)}}{g_c^2(x)};\ee^{\frac{ {(\ln f_w)'(x)}{\ln g_w(x)}- {\ln f_w(x)}{(\ln g_w)'(x)}}{(\ln g_w(x))^2}}\right\rangle\\
  =& \frac{\left\langle{ { f_c'(x)}{ g_c (x)}- { f_c (x)}{ g_c'(x)}};\ee^{{ {(\ln f_w)'(x)}{\ln g_w(x)}- {\ln f_w(x)}{(\ln g_w)'(x)}}}\right\rangle}{\ol g^2(x)}\\
  =& \frac{\left\langle  f_c'(x) g_c(x);\ee^{(\ln f_w)'(x)\ln g_w(x)}\right\rangle-\left\langle f_c(x) g_c'(x);\ee^{\ln f_w(x)(\ln g_w)'(x)}\right\rangle}{\ol g^2(x)}\\
  =&\frac{\overline f'(x)\overline g(x)-\overline f(x)\overline g'(x)}{\overline g(x)^2}.
\end{align*}

The proof is therefore  ended.
\end{proof}

\begin{corollary} Suppose that $\overline f :[a,b]\to \RI$ are differentiable on $[a,b]$, and that $g :[a,b]\to \R $ is differentiable on $[a,b]$. Then for any $x\in[a,b]$,
 \begin{enumerate}
\item[(1)] ${{\left(c_1 \overline f+ c_2  g \right)}{'}}\left( {{x}} \right)=c_1{{\overline f}{'}}\left( {{x}} \right)+c_2{{  g} {'}}\left( {{x} } \right),$ where  $c_1,c_2\in\RI.$

\item[(2)] $(\overline f~  g )'(x)=\overline f'(x)  g(x)+\overline f(x)  g'(x).$

\item[(3)] $\left(\frac{~ \overline f ~}{  g}\right)'(x)=\displaystyle\frac{\overline f'(x)  g(x)-\overline f(x)  g'(x)}{  g^2(x)}.$

\item[(4)]  $\left(\frac{~  g ~}{  \overline f}\right)'(x)=\displaystyle\frac{g'(x)  \overline f(x)- g(x)  \overline f'(x)}{ \overline f^2(x)}.$
  \end{enumerate}
\end{corollary}
\begin{proof}
This is a direct conclusion of Proposition \ref{prop4.5} and Definition \ref{def32222}.
\end{proof}

\begin{theorem}  Suppose that $\overline f:[a,b]\to \RI$ is differentiable on $[a,b]$. Then for any $x\in[a,b]$,
    $${{\left(\overline  f\circ g \right)}{'}}\left( {{x}} \right)={{ \overline f }{'}}\left( {{y}} \right){{g}{'}}\left( {{x}} \right),$$  where $g:[a,b]\to   [a,b]$ is differentiable, and ${{y}}=g\left( {{x}} \right)$.
\end{theorem}

\begin{proof}
 \begin{align*}
 \left(\overline  f\circ g \right)'\left( x \right)=&\lim_{h\to 0}\frac{\ol f(g(x+h))-\ol f(g(x))}{h}\\
 =&\lim_{h\to 0}\left\langle\frac{ f_c(g(x+h))- f_c(g(x))}{h};\left(\frac{f_w(g(x+h))}{f_w(g(x))}\right)^{\frac1h}\right\rangle\\
=&\left\langle f_c'(y)g'(x);\ee^{\frac{f_w'(y)}{f_w (y)} g'(x)}\right\rangle\\
 =&{{ \overline f }{'}}\left( {{y}} \right){{g}{'}}\left( {{x}} \right).
\end{align*}

This ends the proof.
\end{proof}

  \begin{theorem}\label{thmdiff-cont}
Let $\ol f:[a,b]\to \RI$ be   an interval-valued function.  If $\ol f$ is differentiable on $[a,b]$,  then   $\ol f$ is continuous on $[a,b]$.
\end{theorem}

\begin{proof}   Since  $\ol f $ is differentiable on $[a,b]$, then by Theorem \ref{thmdiff},
 $$\overline f'\left( x \right)=\left\langle {f}_c'(x);{f}_w^*(x) \right\rangle.$$
 From  \eqref{mulder},  it follows that
 $${f}_w^*(x) =\ee^{(\ln f_w)'(x)}=\ee^{\frac{f_w'(x)}{f_w(x)}},$$
which implies that $f_w $ is  differentiable on $[a,b]$.
Hence, $f_l=f_c-f_w$ and $f_r=f_c+f_w$ are also  differentiable on $[a,b]$.
Thus,   $f_l$ and $f_r$  are continuous on $[a,b]$, by Theorem \ref{thmcont}, so is $\ol f$.
\end{proof}

\subsection{Comparison with the  $gH$-Derivative}
Now we discuss the relationship between the new derivative and the $gH$-derivative.

\begin{theorem}\label{diff-ghdiff}
 Let $\ol f:[a,b]\to \RI$ be an IVF. If $\ol f$ is continuous on $[a,b]$, then $\ol f $ is differentiable on $[a,b]$ if and only if  $\ol f $ is $gH$-differentiable on $[a,b]$.
\end{theorem}
\begin{proof} (Necessity) Since  $\ol f $ is differentiable on $[a,b]$,  it follows from  the proof of  Theorem \ref{thmdiff-cont} that
$f_l=f_c-f_w$ and $f_r=f_c+f_w$ are also  differentiable on $[a,b]$.
Thus, by \cite[Theorem 17]{SB09}, $\ol f $ is $gH$-differentiable on $[a,b]$.

  (Sufficiency) The sufficiency is obvious according to the necessity.
\end{proof}

\begin{example}\label{ex4}
Let us consider the function $\ol f:\mathbb{R}\rightarrow \RI$,
$$\ol f(x)=\left\{\begin{array}{ll}
{[0,1]}\Big(1-x^2\sin\frac{1}{x}\Big),\quad &\text{if}~x\not=0;\\
{[0,1]},&\text{otherwise}.
\end{array}\right.$$
Since
\begin{align*}
\ol f_{gH}'(x)=\lim\limits_{h\rightarrow0}\frac{\ol f(h)\ominus_{gH}\ol f(0)}{h}
=&\lim\limits_{h\rightarrow0}\frac{{[0,1]}\Big(1-h^2\sin\frac{1}{h}\Big)\ominus_{gH}[0,1]}{h}\\
=&\lim\limits_{h\rightarrow0}{\left[\min\left\{0,-h\sin\frac1h\right\}, \max\left\{0,-h\sin\frac1h\right\}\right]}\\
=&0,
\end{align*}
$\ol f$ is differentiable at $x_0=0$ and $\ol f_{gH}'(0)=0$.
But, \begin{align*}
 \ol f'(x)=\lim\limits_{h\rightarrow0}\frac{\ol f(h) -\ol f(0)}{h}
=&\lim\limits_{h\rightarrow0} \left\langle\frac {-h\sin\frac1h}{2};\left(\frac12\right)^{{(-h\sin \frac1h)}}\right\rangle\\
=&\langle0;1\rangle\\
=&[-1,1].
\end{align*}

\end{example}

\begin{example}\label{e6}
Let us consider the function $f:\mathbb{R}\rightarrow \RI$,
$$f(x)=[-|x|, |x|+1].$$
Since
\begin{align*}
\lim\limits_{h\rightarrow0}\frac{\ol f(h)-\ol f(0)}{h}
=&\lim\limits_{h\rightarrow0 }\frac{\langle \frac12,|h|+\frac12\rangle-\langle\frac12;\frac12 \rangle}{h}\\
=&\lim\limits_{h\rightarrow0 }{\langle 0,\left(2|h|+1\right)^\frac1h\rangle }\\
=&\begin{cases}
 \langle0,\ee^{2}\rangle, & \text{ if } h>0,\\
\langle0,\ee^{-2}\rangle,  & \text{ if } h<0,
\end{cases},
\end{align*}
so, $\ol f$ is not differentiable at $x_0=0$.
 However, $f_l$ and $f_r$ are both not differentiable at $x=0$, so $\ol f$ is neither  not $gH$-differentiable   at $x=0$.
\end{example}

\begin{remark}
  Theorem \ref{diff-ghdiff} is not true, if the condition that $\ol f$ is continuous on $[a,b]$ is omitted, see the  counterexample \ref{e5}.
\end{remark}
\begin{example}\label{e5} \cite{QD}
Consider the IVF $f:\mathbb{R}\rightarrow \RI$ as follows,
$$\ol f(x)=\left\{\begin{array}{ll}
{[-1,x^2+x]}\,\,\,\,\,\,\text{if}~x\in \mathbb{Q};\\
{[x-1,x^2]},\,\,\,\,\,\,\,\,\text{if}~x\in \mathbb{R}\setminus\mathbb{Q}.
\end{array}\right.$$
If $h\in \mathbb{Q}$, then
\begin{align*}
\lim\limits_{h\rightarrow0}\frac{\ol f(h)-\ol f(0)}{h}
=&\lim\limits_{h\rightarrow0}\frac{{[-1,h^2+h]}-[-1,0]}{h}\\
=&\lim\limits_{h\rightarrow0}\left\langle\frac{h+1}{2};(h^2+h+1)^{\frac1h}\right\rangle\\
=&\left\langle\frac12;\ee\right\rangle.
\end{align*}
If $h\in \mathbb{R}\setminus\mathbb{Q}$, then
\begin{align*}
\lim\limits_{h\rightarrow0}\frac{\ol f(h)-\ol f(0)}{h}
=&\lim\limits_{h\rightarrow0}\frac{{[h-1,h^2]}-[-1,0]}{h}\\
=&\lim\limits_{h\rightarrow0}\left\langle\frac{h+1}{2};(h^2-h+1)^{\frac1h}\right\rangle\\
=&\left\langle\frac12;\ee^{-1}\right\rangle.
\end{align*}
Hence, $f$ is not differentiable at $x_0=0$.

On the other hand, $\ol f$ is gH-differentiable at $x_0=0$ and $\ol f'_{gH}(0)=[0,1]$.
\end{example}

\begin{remark}
  It is worth noting that when calculating the $gH$-derivative, the issue of the function's switching points needs to be considered, see \cite{SB09}. However, the new derivative proposed in Definition \ref{def1} does not have this issue, thus the complexity of the calculation is significantly reduced.
\end{remark}

\begin{example}\label{ex415} Considering the IVF $\ol f(x) =\left[ \frac12 x^2, 1 + \frac12 x^2 + 2\sin^2 x\right],$ $x\in[0,2\pi]$. If we calculate the $gH$-derivative $\ol f'_{gH}(x)=[ f'_{l~gH}(x), f'_{r~gH}(x)]$, it has three switching points  $x_1 = \frac12 \pi$,
  $x_2 = \pi$  and $x_3 = \frac32 \pi$.
  However, $$\ol f'(x)=[ f'_{l}(x), f'_{r}(x)]=\left[x+\sin 2x-\ee^{\frac{2\sin 2x}{2-\cos 2x}},x+\sin 2x+\ee^{\frac{2\sin 2x}{2-\cos 2x}}\right],\quad x\in[0,2\pi],$$ see Figure \ref{fig415}.

\begin{figure}[htp]
  \centering
\subfigure[{$\ol f(x)$}]{
\includegraphics[width=0.47\linewidth]{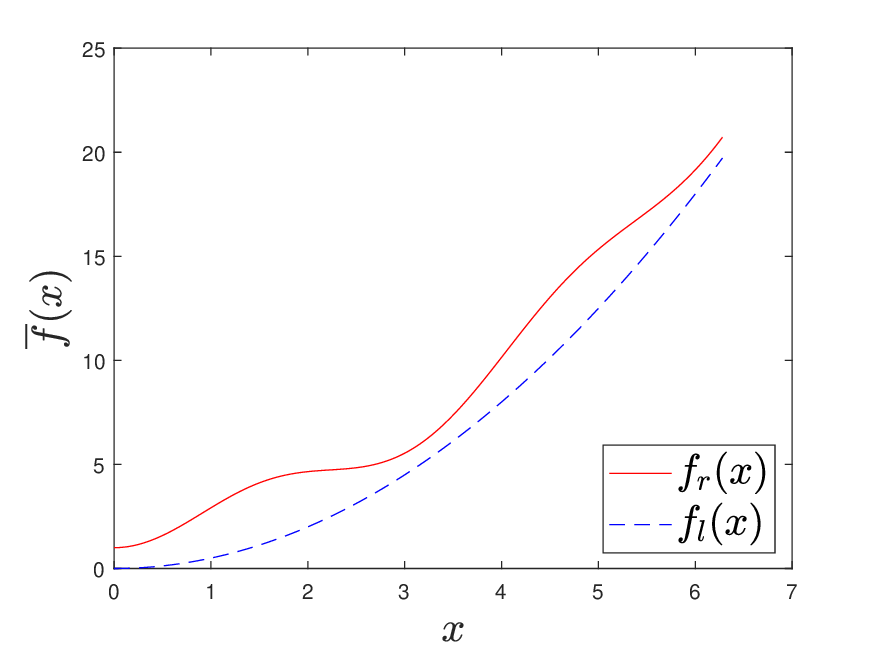}
\label{fig:sub1}
}
  \subfigure[{$\ol f'_{gH}(x)$ and $\ol f'(x)$}]{
\includegraphics[width=0.47\linewidth]{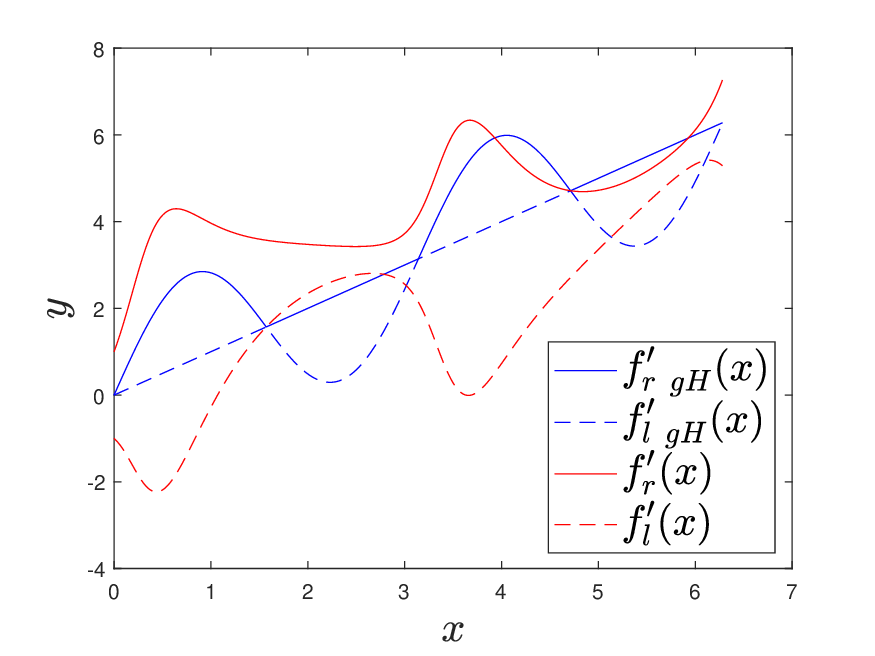}
\label{fig:sub2}
}
   \caption{The IVF $\ol f(x)$ in Example \ref{ex415} and its derivatives}\label{fig415}
\end{figure}
\end{example}

 \section{Integration of IVFs}
In this subsection, we will introduce the concept of Riemann type integrals of interval-valued functions and its properties.

\subsection{Basic Definition and Properties}

\begin{definition}
 Let $\overline f:\left[ a,b \right]\to {{\mathbb{R}}_{\mathcal{I}}},\ol A\in {{\mathbb{R}}_{\mathcal{I}}}$. If for arbitrary $\varepsilon >0$, there exists a constant $\delta >0$ such that for any division $P$ of $\left[ a,b \right]$ given by
 $$a={{t}_{0}}<{{t}_{1}}<{{t}_{2}}<\cdots <{{t}_{n-1}}<{{t}_{n}}=b ,\text{
 and }(\xi_1,\xi_2,...,\xi_n)$$ satisfies $t_{i-1}\leq \xi_i\leq t_i$
 and $t_i-t_{i-1}<\delta $ for all $i$, one has
$$d\left(\sum_{i=1}^{n} {\overline f\left( {{\xi }_{i}} \right)\left( {{t}_{i}}-{{t}_{i-1}} \right)},\ol A \right)<\varepsilon.$$
Then  $\overline  f$ is said to be interval Riemann integrable to $\ol A$ on $\left[ a,b \right]$,  and denoted by $$\ol A=\left( IR \right)\int_{a}^{b}{\overline f\left( t \right) }\dd t.$$
\end{definition}
Denote by $\mathcal{IR}([a,b],\RI)$ the space of all  interval Riemann integrable IVFs $\ol f$ on $[a,b]$. Also,  denote by $\mathcal{R}([a,b],\R)$ the space of all  Riemann integrable funcitons $ f$ on $[a,b]$.

\begin{theorem}\label{int-1}
Suppose that $\overline  f:\left[ a,b \right]\to {{\mathbb{R}}_{\mathcal{I}}} $ with $\overline  f\left( t \right)=\left\langle  {f}_c\left( t \right), {f}_w\left( t \right) \right\rangle,t\in \left[ a,b \right] $, then $\overline  f\in \mathcal{I}\mathcal{R}\left( \left[ a,b \right] ,\RI\right)$ if and only if $ {f}_c,\ln{f}_w\in \mathcal{R}\left( \left[ a,b \right],\R  \right)$. Moreover,
 $$\left( IR \right)\int_{a}^{b}{\overline f\left( t \right) }\dd t=\left\langle \left( R \right)\int_{a}^{b}{ {f}_c\left( t \right) }\dd t;\left( R \right)\int_{a}^{b}{ {f}_w\left( t \right)}^{\dd t} \right\rangle, $$
 where $\left( R \right)\int_{a}^{b}{ {f}_w\left( t \right)}^{\dd t}   $ is the multiplication integral of   ${f}_w$ on $[a,b]$  (see details in \cite{BKO08}), i.e.,
 \begin{eqnarray}\label{MI}
   ( R  )\int_{a}^{b}{ {f}_w\left( t \right)}^{\dd t} =\ee^{( R  ) \int_{a}^{b}(\ln f_w)(t)\dd t}.
 \end{eqnarray}
\end{theorem}
\begin{proof}By Definition \ref{def1}, one has
\begin{eqnarray*}
\left( IR \right)\int_{a}^{b}{\overline f\left( t \right) }\dd t
&=&\lim_{\delta\to 0}\sum\limits_{i=1}^{n}{\overline f\left( {{\xi }_{i}} \right)\left( {{t}_{i}}-{{t}_{i-1}} \right)}\\ &=&\lim_{\delta\to 0}\sum\limits_{i=1}^{n}{\langle f_c\left( {{\xi }_{i}} \right);f_w(\xi_i)\rangle\left( {{t}_{i}}-{{t}_{i-1}} \right)}\\
&=&\lim_{\delta\to 0}{\left\langle \sum\limits_{i=1}^{n} f_c\left( {{\xi }_{i}} \right)\left( {{t}_{i}}-{{t}_{i-1}} \right);\prod_{i=1}^n f_w(\xi_i)^{\left( {{t}_{i}}-{{t}_{i-1}} \right)}\right\rangle}\\
&=&\lim_{\delta\to 0}{\left\langle \sum\limits_{i=1}^{n} f_c\left( {{\xi }_{i}} \right)\left( {{t}_{i}}-{{t}_{i-1}} \right);\ee^{\sum\limits_{i=1}^n \ln f_w(\xi_i) {\left( {{t}_{i}}-{{t}_{i-1}} \right)}}\right\rangle}\\
&=&\left\langle \left( R \right)\int_{a}^{b}{ {f}_c\left( t \right) }\dd t;\ee^{( R  ) \int_{a}^{b}(\ln f_w)(t)\dd t}\right\rangle,\\
&=&\left\langle \left( R \right)\int_{a}^{b}{ {f}_c\left( t \right) }\dd t;\left( R \right)\int_{a}^{b}{ {f}_w\left( t \right)}^{\dd t} \right\rangle.
\end{eqnarray*}
This ends the proof.
\end{proof}

\begin{theorem}\label{int-2}
Suppose that $\overline  f\in  C(\left[ a,b \right], {{\mathbb{R}}_{\mathcal{I}}}) $, then $\overline  f\in \mathcal{I}\mathcal{R}\left( \left[ a,b \right] ,\RI\right)$.
\end{theorem}

\begin{proof} Since $\overline  f\in  C(\left[ a,b \right], {{\mathbb{R}}_{\mathcal{I}}}) $, by Theorem \ref{thmcont},  $f_l$ and $f_r$ are continuous on $[a, b]$, so are $f_c$ and $f_w$. Hence, $f_c$ and $\ln f_w$ are integrable on $[a,b]$. From Theorem \ref{int-1}, it follows that
$\overline  f\in \mathcal{I}\mathcal{R}\left( \left[ a,b \right] ,\RI\right)$.
\end{proof}

The interval Riemann integral has the following properties.

\begin{proposition} \label{proprop}

   (1) {\bf  (Monotonicity)}.  Suppose that $\overline f,\overline g\in \mathcal{I}\mathcal{R}\left( \left[ a,b \right],\RI \right)$. If $\overline f\left( t \right)\leq \overline g\left( t \right) $ for any $ t\in \left[ a,b \right] $,  then
 $$ \left( IR \right)\int_{a}^{b}{\overline f\left( t \right) }\dd t\leq \left( IR \right)\int_{a}^{b}{\overline g\left( t \right) }\dd t. $$

 (2)  {\bf (Linearity)}. Suppose that  $\overline f,\overline g\in \mathcal{I}\mathcal{R}\left( \left[ a,b \right] ,\RI \right),\lambda_1,\lambda_2 \in \RI$, then
  $\lambda_1\overline f+\lambda_2\overline g\in \mathcal{I}\mathcal{R}\left( \left[ a,b \right] ,\RI\right)$, and
 $$ \left( IR \right)\int_{a}^{b}{\left[\lambda_1 \overline f\left( t \right)\ +\lambda_2 \overline g\left( t \right) \right]}\dd t=\lambda_1\left( IR \right)\int_{a}^{b}{\overline f\left( t \right) }\dd t+\lambda_2\left( IR \right)\int_{a}^{b}{\overline g\left( t \right) }\dd t. $$

 (3)  {\bf (Subinterval Integrability)}. If IVF $\overline f\in \mathcal{I}\mathcal{R}\left( \left[ a,b \right] ,\RI\right)$, then  so
 it is on a subinterval
  $ \left[ c,d \right]\subseteq \left[ a,b \right] $.

 (4) {\bf (Additivity)}. Suppose that $\overline f\in \mathcal{I}\mathcal{R}\left( \left[ a,b \right],\RI \right)$, then for any   $c\in \left[ a,b \right]$,  it holds
 $$\left( IR \right)\int_{a}^{b}{\overline f\left( t \right) }\dd t=\left( IR \right)\int_{a}^{c}{\overline f\left( t \right) }\dd t+\left( IR \right)\int_{c}^{b}{\overline f\left( t \right) }\dd t. $$
\end{proposition}

\begin{proof} (1) By Definition \ref{TOrder}, there are three cases in ``$\leq$".

(I) If  $\overline f\left( t \right)= \overline g\left( t \right) $, $t\in[a,b]$,  then it is easy to see that
 $$ \left( IR \right)\int_{a}^{b}{\overline f\left( t \right) }\dd t= \left( IR \right)\int_{a}^{b}{\overline g\left( t \right) }\dd t. $$

(II)  If  $ f_c\left( t \right)<  g_c\left( t \right) $, $t\in[a,b]$,  then one has
 $$ \left( R \right)\int_{a}^{b}{  f_c\left( t \right) }\dd t<\left( R \right)\int_{a}^{b}{  g_c\left( t \right) }\dd t, $$
that is,
 $$ \left( IR \right)\int_{a}^{b}{\overline f\left( t \right) }\dd t\leq \left( IR \right)\int_{a}^{b}{\overline g\left( t \right) }\dd t. $$

(III)  If $ f_c\left( t \right)=  g_c\left( t \right) $,  $ f_w\left( t \right)<  g_w\left( t \right) $, $t\in[a,b]$,  then we have
 $$ \left( R \right)\int_{a}^{b}{ \ln f_w\left( t \right) }\dd t<\left( R \right)\int_{a}^{b}{ \ln g_w\left( t \right) }\dd t, $$
by \eqref{MI}, it holds
 $$ \left( IR \right)\int_{a}^{b}{\overline f\left( t \right) }\dd t\leq \left( IR \right)\int_{a}^{b}{\overline g\left( t \right) }\dd t. $$

(2) Since  $\overline f,\overline g\in \mathcal{I}\mathcal{R}\left( \left[ a,b \right] ,\RI \right) $, by Lemma \ref{int-1}, the functions $f_c,\ln f_w,g_c,\ln g_w \in \mathcal{R}\left( \left[ a,b \right] ,\RI \right),$ so are $\lambda_1  f_c+\lambda_2  g_c$ and $\lambda_1  \ln f_w+\lambda_2  \ln g_w$. Then
$$\lambda_1\overline f+\lambda_2\overline g=\left\langle \lambda_1  f_c+\lambda_2  g_c;  \ee^{\lambda_1  \ln f_w+\lambda_2  \ln g_w}\right\rangle\in \mathcal{I}\mathcal{R}\left( \left[ a,b \right] ,\RI \right).$$
Moreover,
\begin{align*}
&\left( IR \right)\int_{a}^{b}{\left[\lambda_1 \overline f\left( t \right)\ +\lambda_2 \overline g\left( t \right) \right]}\dd t\\
=&(IR)\int_{a}^{b}\left\langle \lambda_1  f_c(t)+\lambda_2  g_c(t);  \ee^{\lambda_1  \ln f_w(t)+\lambda_2  \ln g_w(t)}\right\rangle\dd t\\
=&\left\langle( R)\int_{a}^{b} \lambda_1  f_c(t)+\lambda_2  g_c(t)\dd t;  \ee^{( R)\int_{a}^{b}\lambda_1  \ln f_w(t)+\lambda_2  \ln g_w(t) \dd t}\right\rangle\\
=&\left\langle  \lambda_1( R)\int_{a}^{b}  f_c(t)\dd t+\lambda_2 ( R)\int_{a}^{b} g_c(t)\dd t;  \ee^{\lambda_1  ( R)\int_{a}^{b}\ln f_w(t)\dd t+\lambda_2 ( R)\int_{a}^{b} \ln g_w(t) \dd t}\right\rangle\\
=&\left\langle  \lambda_1( R)\int_{a}^{b}  f_c(t)\dd t;  \ee^{\lambda_1  ( R)\int_{a}^{b}\ln f_w(t)\dd t}\right\rangle+
            \left\langle  \lambda_2 ( R)\int_{a}^{b} g_c(t)\dd t;  \ee^{\lambda_2 ( R)\int_{a}^{b} \ln g_w(t) \dd t}\right\rangle\\
=&\lambda_1\left( IR \right)\int_{a}^{b}{\overline f\left( t \right) }\dd t+\lambda_2\left( IR \right)\int_{a}^{b}{\overline g\left( t \right) }\dd t.
\end{align*}

(3) Since  $\overline f\in \mathcal{I}\mathcal{R}\left( \left[ a,b \right] ,\RI \right) $, by Lemma \ref{int-1}, the functions $f_c,\ln f_w \in \mathcal{R}\left( \left[ a,b \right] ,\RI \right).$ Hence,  for any $[c,d]\subseteq [a,b]$, the functions $f_c,\ln f_w \in \mathcal{R}\left( \left[ c, d \right] ,\RI \right)$, thereby,  $\overline f\in \mathcal{I}\mathcal{R}\left( \left[ c,d \right] ,\RI \right) $.

(4) Since  $\overline f\in \mathcal{I}\mathcal{R}\left( \left[ a,b \right] ,\RI \right) $, by Lemma \ref{int-1}, the functions $f_c,\ln f_w \in \mathcal{R}\left( \left[ a,b \right] ,\RI \right).$ Hence,  for any $c\in [a,b]$, the functions $f_c,\ln f_w \in \mathcal{R}\left( \left[ a, c \right] ,\RI \right)$, and also $f_c,\ln f_w \in \mathcal{R}\left( \left[ c, b \right] ,\RI \right)$, thereby,
\begin{align*}&\left( IR \right)\int_{a}^{b}{\overline f\left( t \right) }\dd t\\
=&
\left\langle(R)\int_{a}^{b}f_c(t)\dd t; \ee^{(R)\int_{a}^{b}\ln f_w(t)\dd t}\right\rangle\\
=&
\left\langle(R)\int_{a}^{c}f_c(t)\dd t+(R)\int_{c}^{b}f_c(t)\dd t; \ee^{(R)\int_{a}^{c}\ln f_w(t)\dd t+(R)\int_{c}^{b}\ln f_w(t)\dd t}\right\rangle\\
=&
\left\langle(R)\int_{a}^{c}f_c(t)\dd t; \ee^{(R)\int_{a}^{c}\ln f_w(t)\dd t}\right\rangle+\left\langle(R)\int_{c}^{b}f_c(t)\dd t; \ee^{(R)\int_{c}^{b}\ln f_w(t)\dd t}\right\rangle\\
=&\left( IR \right)\int_{a}^{c}{\overline f\left( t \right) }\dd t+\left( IR \right)\int_{c}^{b}{\overline f\left( t \right) }\dd t.
\end{align*}
This ends the proof.
\end{proof}

By  Proposition \ref{proprop} and Definition \ref{def32222}, we have
\begin{corollary}  Suppose that  $\overline f \in \mathcal{I}\mathcal{R}\left( \left[ a,b \right] ,\RI \right) $ and $  g\in \mathcal{R}\left( \left[ a,b \right] ,\R  \right),$
  $\lambda_1,\lambda_2 \in \RI$, then
  $\lambda_1\overline f+\lambda_2  g\in \mathcal{I}\mathcal{R}\left( \left[ a,b \right] ,\RI\right)$, and
 $$ \left( IR \right)\int_{a}^{b}{\left[\lambda_1 \overline f\left( t \right)\ +\lambda_2   g\left( t \right) \right]}\dd t=\lambda_1\left( IR \right)\int_{a}^{b}{\overline f\left( t \right) }\dd t+\lambda_2\left( R\right)\int_{a}^{b}{  g\left( t \right) }\dd t. $$
\end{corollary}

\subsection{Fundamental Theorem of Calculus}

\begin{theorem}\label{FTC-1}
Let $\ol f \in \mathcal{IR}([a,b],\RI)$, if $\ol F:[a,b]\to \RI$ is given by
$$\ol F(t)=(IR)\int_{a}^{t} \ol f(s)\dd s, ~~ t\in [a,b],$$
then
$\ol F  $ is continuous on $[a,b]$ and $\ol  F'=\ol f$  at all points
of continuity of $\ol f$.
\end{theorem}
\begin{proof} Since  $\overline f\in \mathcal{I}\mathcal{R}\left( \left[ a,b \right] ,\RI \right) $,  and
$$
\ol F(t)=(IR)\int_{a}^{t} \ol f(s) \dd s , ~~   t \in [a, b],
$$
then   for any $t_1,t_2\in[a,b]$,
\begin{align*}
  d(\ol F(t_1),\ol F(t_2))=&d\left((IR)\int_{a}^{t_1} \ol f(s) \dd s,(IR)\int_{a}^{t_2} \ol f(s) \dd s\right)\\
    =&\sqrt{\left((R)\int_{t_2}^{t_1}   f_c(s) \dd s \right)^2+\left((R)\int_{t_2}^{t_1}  \ln f_w(s) \dd s \right)^2}\\
    \to& 0, \text{ as } t_1\to t_2,
\end{align*}
so
$\ol F  $ is continuous on $[a,b]$.

Moreover,  for any point $x$
of continuity of $\ol f$ and $x+h\in[a,b]$,
\begin{align*}
 \ol  F'(x)=& \lim_{h\to0}\frac{\ol F(x+h)-\ol F(x )}{h}\\
 =& \lim_{h\to0}\frac{(IR)\int_{x}^{x+h} \ol f(s) \dd s}{h} \\
    =& \lim_{h\to0}\ol f(\xi)\\
    =&\ol f(x).
\end{align*}
This ends the proof.
\end{proof}

\begin{theorem} \label{tnb}
Assume that $\ol f\in C([a,b], \RI)$. If $\ol F$ is  a primitive of $\ol f$ on $[a,b]$, then
\begin{equation} \label{nb}
\ol F(b) - \ol F(a)=(IR)\int_{a}^{b}\ol f  (t)\dd t.
\end{equation}
\end{theorem}
\begin{proof} Since $\ol f\in C([a,b], \RI)$, by Theorem \ref{FTC-1}, $( IR)\int_{a}^{t}\ol f(s)\dd s  $ is   a primitive of $\ol f$ on $[a,b]$. This and the fact that $\ol F$ is  a primitive of $\ol f$ on $[a,b]$ imply that
\begin{align}\label{eqNL-1}
  F(t)=( IR)\int_{a}^{t}\ol f(s)\dd s+\ol C, ~~t\in[a,b],~~\ol C\in\RI.
\end{align}
Taking $t=a$ and $t=b$ in \eqref{eqNL-1}, on has
$$\ol F(b) - \ol F(a)=(IR)\int_{a}^{b}\ol f  (t)\dd t.$$
The proof is therefore completed.
\end{proof}

The following example shows that (\ref{nb}) is strict.
\begin{example} Considering that  $\ol F(t)=[t,t^2+1]$, $t\in [0,1]$. By Theorem \ref{thmdiff}, we have
$$\ol F'(t)=\left\langle t+\frac12;\ee^{\frac{2t-1}{t^2-t+1}}\right\rangle,~~t\in[0,1].$$
Due to Theorem \ref{int-1}, one has
\begin{align*}
  (IR)\int_0^1\ol  F'(t)\dd t=&(IR)\int_0^1\left\langle t+\frac12;\ee^{\frac{2t-1}{t^2-t+1}}\right\rangle\dd t\\
  =& \left\langle 1;1\right\rangle\\
  =&[0.2].
\end{align*}
On the other hand,
$$\ol F(1)-\ol F(0)=[1,2]-[0,1]=[0,2].$$
Hence, $$\ol F(1)-\ol F(0)=(IR)\int_0^1\ol  F'(t)\dd t.$$

However,  it is easy to see that
$$
\ol F'_{gH}=\begin{cases}
  [2t,1],&  t\in[0,\frac12],\\
  [1,2t],& t\in[\frac12,1].
\end{cases}
$$
Therefore,
$$
 \int_0^1 \ol F'_{gH}(t)\dd t=  \int_{0}^{\frac{1}{2}}[2 t, 1] \dd t+ \int_{\frac{1}{2}}^{1}[1,2 t]\dd t=\left[\frac{3}{4},\frac{5}{4}\right].
$$
But
$$
\ol F(1) \ominus_{gH} \ol F(0)=[1,1]\neq  \int_0^1 \ol F'_{gH}(t)\dd t.
$$
\end{example}

\subsection{Integration by Parts}

Let  $C^1([a,b], \RI)$ be the space of continuously differentiable IVFs on $[a,b]$, and
$C^1([a,b], \mathbb{R})$   the space of continuously differentiable functions on $[a,b]$.

\begin{lemma}\label{lem3.313} Let $\ol f,\ol g \in \mathcal{IR}([a, b],\RI)$. Then $\ol f~\ol g\in \mathcal{IR}([a, b] ,\RI).$
\end{lemma}

\begin{proof} Since  $\overline f,\overline g\in \mathcal{I}\mathcal{R}\left( \left[ a,b \right] ,\RI \right) $, by Lemma \ref{int-1}, the functions $f_c,\ln f_w,g_c,\ln g_w \in \mathcal{R}\left( \left[ a,b \right] ,\RI \right),$ so are $  f_c  g_c$ and $ \ln f_w  \ln g_w$. Then
$$ \overline f \overline g=\left\langle    f_c   g_c;  \ee^{ \ln f_w    \ln g_w}\right\rangle\in \mathcal{I}\mathcal{R}\left( \left[ a,b \right] ,\RI \right).$$ The proof is therefore complete.
\end{proof}
\begin{corollary}\label{cor3.313} Let $\ol f \in \mathcal{IR}([a, b],\RI)$ and $g\in \mathcal{ R}([a, b],\R ) $. Then $\ol f g\in \mathcal{IR}([a, b] ,\RI).$
\end{corollary}

\begin{theorem} \label{fb2}
Let $\ol F,~\ol G\in C^1([a,b], \RI)$.
Then we have
\begin{equation}\label{eq31}
\ol F(b)\ol  G(b) - \ol F(a)\ol  G(a) = (IR)\int_{a}^{b}\ol  F'(t)\ol  G(t)\dd t +(IR)\int_{a}^{b}\ol  F(t) \ol G'(t)\dd t.
\end{equation}
\end{theorem}

\begin{proof} Since $\ol F,~\ol G\in C^1([a,b], \RI)$, $\ol F',~\ol G'\in C ([a,b], \RI)$, so $\ol  F' ~\ol  G$ and $\ol  F ~\ol  G'$  are continuous, and hence
 $\ol  F' ~\ol  G,~~\ol  F ~\ol  G'\in \mathcal{IR}([a,b], \RI).$
Thus,
\begin{align*}
&(IR)\int_{a}^{b}\ol  F'(t)\ol  G(t)\dd t +(IR)\int_{a}^{b}\ol  F(t) \ol G'(t)\dd t\\
=&\left\langle(R)\int_{a}^{b} {F}_c'(t)G_c(t)\dd t,\ee^{(R)\int_{a}^{b}\ln{F}_w^*(t)\ln G_w(t)\dd t}\right\rangle\\
& +\left\langle(R)\int_{a}^{b} {F}_c(t)G_c'(t)\dd t,\ee^{(R)\int_{a}^{b}\ln{F}_w(t)\ln G_w^*(t)\dd t}\right\rangle\\
 =& \left\langle(R)\int_{a}^{b} {F}_c'(t)G_c(t)\dd t+ (R)\int_{a}^{b} {F}_c(t)G_c'(t)\dd t;\right.\\
& \left.\ee^{(R)\int_{a}^{b}\ln{F}_w^*(t)\ln G_w(t)\dd t+(R)\int_{a}^{b}\ln{F}_w(t)\ln G_w^*(t)\dd t}\right\rangle\\
=&\left\langle F_c'(b)G_c(b)-F_c'(a)G_c(a); \ee^{ \ln{F}_w^*(b)\ln G_w(b)-\ln{F}_w(a)\ln G_w^*(a) }\right\rangle\\
=&\left\langle F_c'(b)G_c(b); \ee^{ \ln{F}_w^*(b)\ln G_w(b)}\right\rangle-
\left\langle F_c'(a)G_c(a); \ee^{ \ln{F}_w(a)\ln G_w^*(a) }\right\rangle\\
=&\left\langle F_c'(b) ;    {F}_w^*(b) \right\rangle\left\langle G_c(b);  G_w(b) \right\rangle-
\left\langle F_c (a) ;    {F}_w(a) \right\rangle\left\langle G_c'(a);  G_w^*(a) \right\rangle\\
=&\ol F(b)\ol  G(b) - \ol F(a)\ol  G(a).
\end{align*}
Therefore,  \eqref{eq31} holds. This ends the proof.
\end{proof}

\begin{corollary} \label{fb1}
Let $\ol F \in C^1([a, b], \RI)$   and $G\in C^1([a,b], \mathbb{R})$.
Then we have
\begin{equation}
\ol  F(b) G(b) -\ol  F(a) G(a)=(IR)\int_{a}^{b}\ol  F(t) G'(t) \dd t+(IR)\int_{a}^{b}\ol  F'(t) G(t)\dd t.
\end{equation}
\end{corollary}

\begin{example}\label{ex31} Considering that  $\ol F(t)=\left[  t^2, 2t+1\right]$, $\ol G(t)=\left[ t,  t^2+ 1\right]$,  $t\in [0,1]$. Then,
\begin{align*}
&(IR)\int_{0}^{1}\ol F(t) \ol  G'(t)\dd t +(IR)\int_{0}^{1}\ol F'(t) \ol  G(t)\dd t \\
=&(IR)\int_{0}^{1}\left\langle \frac{ t^2+ 2t+1}{2}; \frac{ -t^2+ 2t+1}{2} \right\rangle \left\langle t+\frac12;\ee^{\frac{2t-1}{t^2-t+1}}\right\rangle \dd t\\
&+(IR)\int_{0}^{1} \left\langle t+1;\ee^{\frac{-2t+2}{-t^2+2t+1}}\right\rangle \left\langle \frac{ t^2+ t+1}{2}; \frac{ t^2-t+1}{2} \right\rangle \dd t
\\
 =&(IR)\int_{0}^{1}\left\langle \frac{(t+1)^2(2t+1)}{4};\ee^{\frac{2t-1}{t^2-t+1}\ln\frac{ -t^2+ 2t+1}{2}}\right\rangle \dd t\\
 &+(IR)\int_{0}^{1}\left\langle \frac{(t+1)(t^2+t+1)}{2};\ee^{\frac{-2t+2}{-t^2+2t+1}\ln\frac{ t^2-t+1}{2}}\right\rangle \dd t\\
 =&\left\langle (R)\int_{0}^{1}\left(t^{3}+\frac{9}{4}{}t^{2}+2{}t +\frac{3}{4}\right)\dd t ;\ee^{(R)\int_{0}^{1}(\ln\frac{ -t^2+ 2t+1}{2} \ln\frac{ t^2-t+1}{2})'\dd t}\right\rangle\\
  =&\left\langle \frac{11}{4};2^{-\ln 2}\right\rangle,\\
\end{align*}
and
\begin{align*}
  \ol F(1)\ol G(1)- \ol F(0)\ol G(0)=&[1,3][1,2]-[0,1][0,1]\\
  =&\left\langle 3;1\right\rangle-\left\langle \frac14;2^{\ln 2} \right\rangle\\
   =&\left\langle \frac{11}{4};2^{-\ln 2}\right\rangle,
\end{align*}
Thus,
$$\ol F(1)\ol G(1)- \ol F(0)\ol G(0)=(IR)\int_{0}^{1}\ol F(t) \ol  G'(t)\dd t +(IR)\int_{0}^{1}\ol F'(t) \ol  G(t)\dd t .$$
\end{example}

\begin{example}\label{ex32} Considering that $\ol F(t)=\left[  t^2, 2t+1\right]$, and $G(t)=2-t$,   $ t\in [0,1]$. It is easy to see that
  $$(IR)\int_{0}^{1}\ol F(t)G'(t)\dd t=(IR)\int_{0}^{1}\left[  t^2, 2t+1\right](-1)  \dd t =\left\langle-\frac76;\ee^{2+\ln 2-2\sqrt{2}~ \text{arctanh} \frac{\sqrt{2}}{2} }\right\rangle,$$
 $$(IR)\int_{0}^{1}\ol F'(t)  G(t)\dd t=(IR)\int_{0}^{1}\left\langle t+1;\ee^{\frac{-2t+2}{-t^2+2t+1}}\right\rangle (2-t)   \dd t=\left\langle\frac{13}{6}, \ee^{-2+\ln 2+2\sqrt{2}~ \text{arctanh} \frac{\sqrt{2}}{2}}\right\rangle,$$
 hence,
 $$(IR)\int_{0}^{1}\ol F(t)G'(t)\dd t+(IR)\int_{0}^{1}\ol F'(t)  G(t)\dd t=\left\langle1; 4\right\rangle.$$
 On the other hand,
$$\ol F (1)G (1)- \ol F (0)G(0)=1\left[1,3\right]- 2[0,1]=\left\langle1; 4\right\rangle.$$
Thus,
$$\ol F (1)G (1)- \ol F (0)G(0)=(IR)\int_{0}^{1}\ol F(t)G'(t)\dd t+(IR)\int_{0}^{1}\ol F'(t)  G(t)\dd t.$$

\end{example}

\subsection{Convergence theorems}

First, we prove the following monotone convergence theorem.
\begin{theorem}\label{TThem4.1}
If $\ol f_n\in \mathcal{IR}([a,b],\RI)$, $n=1,2,...,$ satisfying\\
(i) $\lim_{n\rightarrow\infty} \ol f_{n}(t)=\ol f(t)$ in $[a, b]$;\\
(ii) $\ol f_{1}(t)\preccurlyeq \ol  f_{2}(t)\preccurlyeq ...\preccurlyeq \ol  f_{n}(t)\preccurlyeq ...$ for all $t\in [a, b]$;\\
(iii) $(IR)\int_a^b \ol f_{n}(t) \dd t$ converges to $\ol A\in\RI$ as $n\rightarrow\infty$.\\
Then $\ol f\in \mathcal{IR}([a,b],\RI)$ and
$$(IR)\int_{a}^{b}\ol f(t)\dd t=\ol A.$$
\end{theorem}

\begin{proof} Let $\ol f(t)=\langle {f}_c(t);{f}_w(t)\rangle$, $\ol f_n(t)=\langle ({f}_n)_c(t); ({f}_n)_w(t)\rangle$, $t\in[a,b]$, $n=1,2,..., $ $\ol A=\langle A_c;A_w\rangle$.
Since $\ol f_n\in \mathcal{IR}([a,b],\RI)$, $n=1,2,...,$ by Lemma \ref{int-1}, one has
$$ ({f}_n)_c, \ln({f}_n)_w\in \mathcal{ R}([a,b],\mathbb{R}) ,~~ n=1,2,... .$$
It follows from (i), (iii) and Theorem  \ref{thmlim} that
$$ \lim_{n\rightarrow\infty}  ({f}_{n})_c(t)=  {f}_c(t),~~~~ \lim_{n\rightarrow\infty}  ({f}_{n})_w(t)=  {f}_w(t),~~ t\in [a, b],$$
and
$$ \lim_{n\rightarrow\infty}(R)\int_a^b  ({f}_{n})_c(t)\dd t=A_c,~~~~\lim_{n\rightarrow\infty}(R)\int_a^b \ln ({f}_{n})_w(t)\dd t=\ln A_w.$$
Moreover, by (ii) and Definition \ref{defpreceq}, we have
$$ ({f}_{1})_c(t)\leq  ({f}_{2})_c(t)\leq...\leq  ({f}_{n})_c(t)\leq...,~~ t\in [a, b],$$
and$$ ({f}_{1})_w(t)\leq ({f}_{2})_w(t)\leq...\leq ({f}_{n})_w(t)\leq..., ~~ t\in [a, b],$$
i.e.,
$$ \ln({f}_{1})_w(t)\leq\ln ({f}_{2})_w(t)\leq...\leq \ln({f}_{n})_w(t)\leq..., ~~ t\in [a, b].$$
According to \cite[Theorem 4.1]{PYL89}, it holds that $ {f}_c, {f}_w\in \mathcal{R}([a,b],\mathbb{R})$ and
$$ (R)\int_{a}^{b} {f}_c(t)\dd t= A_c,~~~~ (R)\int_{a}^{b}\ln {f}_w(t)\dd t= \ln A_w.$$
This implies that
$$(IR)\int_{a}^{b}\ol f(t)\dd t=\ol A.$$
 Therefore, the proof is completed.
\end{proof}

Now we state and prove the dominated convergence theorem.

\begin{theorem}\label{DCT}
If $\ol f_n\in \mathcal{IR}([a,b],\RI)$, $n=1,2,...,$ satisfying\\
(i) $\lim_{n\rightarrow\infty}\ol f_{n}(t)=\ol f(t)$ in $[a, b]$;\\
(ii) $\ol g(t)\preccurlyeq \ol f_{n}(t)\preccurlyeq \ol h(t)$ for all $t\in[a, b]$ and all $n$, where $\ol g,\ol h\in \mathcal{IR}([a,b],\RI)$.\\
 Then $\ol f\in \mathcal{IR}([a,b],\RI)$ and
$$ \lim_{n\rightarrow\infty}(IR)\int_{a}^{b}\ol f_{n}(t)\dd t= (IR)\int_{a}^{b}\ol f(t)\dd t.$$
\end{theorem}
\begin{proof}Let Let $\ol f(t)=\langle {f}_c(t);{f}_w(t)\rangle$, $\ol f_n(t)=\langle ({f}_n)_c(t); ({f}_n)_w(t)\rangle$, $t\in[a,b]$, $n=1,2,..., $  $\ol g(t)=\langle {g}_c(t);{g}_w(t)\rangle$, $\ol h(t)=\langle {h}_c(t);{h}_w(t)\rangle$, $t\in[a,b]$.
Since $\ol g,\ol h,\ol f_n\in \mathcal{IR}([a,b],\RI)$, $n=1,2,...,$ by Lemma \ref{int-1}, one has
$$ {g}_c, \ln {g}_w, {h}_c, \ln {h}_w, ({f}_n)_c, \ln ({f}_n)_w\in \mathcal{R}([a,b],\mathbb{R}) ,~~ n=1,2,... .$$
Due to (i) and Theorem
\ref{thmlim}, we have
$$ \lim_{n\rightarrow\infty}  ({f}_{n})_c(t) =  {f}_c(t),$$$$ \lim_{n\rightarrow\infty}  \ln ({f}_{n})_w(t)= \ln {f}_w(t),~~ t\in [a, b].$$
Moreover, thanks to (ii) and Definition \ref{defpreceq}, we get
$$ {g}_c (t)\leq  ({f}_{n})_c(t)\leq  {h}_c(t),$$$$\ln {g}_w (t)\leq  \ln ({f}_{n})_w(t)\leq  \ln {h}_w(t), ~~ t\in [a, b].$$
By \cite[Theorem 4.3]{PYL89}, it holds that $ {f}_c, {f}_w\in \mathcal{R}([a,b],\mathbb{R})$ and
$$ \lim_{n\rightarrow\infty}(R)\int_{a}^{b} ({f}_{n})_c(t)\dd t= (R)\int_{a}^{b} {f}_c(t)\dd t,$$$$ \lim_{n\rightarrow\infty}(R)\int_{a}^{b}\ln ({f}_{n})_w(t)\dd t= (R)\int_{a}^{b} \ln {f}_w(t)\dd t.$$
This implies that
$$ \lim_{n\rightarrow\infty}(IR)\int_{a}^{b}\ol f_{n}(t)\dd t= (IR)\int_{a}^{b}\ol f(t)\dd t.$$
The proof is therefore completed.
\end{proof}

\section{Interval Differential Equations}

Now, considering the following interval  differential equation
\begin{equation}\label{L1.1}
\left\{\begin{array}{ll}
 \ol x'=\ol f(t,\ol x),\\
 \ol x(a)=\ol x_0,
\end{array}\right.
\end{equation}
where $\ol x_0\in \RI$
and $\ol f:[a,b]\times\mathbb{R}_\mathcal{I}\rightarrow \mathbb{R}_\mathcal{I}$,  $\ol x'$ denotes the derivative of $\ol x\in C^1([a,b],\mathbb{R}_\mathcal{I}).$

\begin{definition}
Suppose $\ol x\in C^1([a,b],\mathbb{R}_\mathcal{I})$. We say that $\ol x$ is a solution of Equation \eqref{L1.1} if $\ol x$ satisfies \eqref{L1.1}.
\end{definition}

We now impose some assumptions on  $\ol f$.
\begin{itemize}
 \item[($H_1$)]~$\ol f(t,\cdot)$ is continuous for all $t\in [a,b]$.
 \item[($H_2$)]~$\ol f(\cdot, \ol x(\cdot))$ is Riemann integrable for every fixed $\ol x\in C^1 ([a,b],\mathbb{R}_\mathcal{I})$.
 \item[($H_3$)] There exist  two Riemann integrable IVFs $\ol g$ and $\ol h$
 on $[a,b]$, such that
 $$\ol g(t)\preccurlyeq \ol  f(t,\ol x(t))\preccurlyeq \ol  h(t),~~\forall ~\ol x\in C^1 ([a,b],\mathbb{R}_\mathcal{I}), ~~t\in [a,b].$$
\end{itemize}

Based on Theorem \ref{tnb}, we have

\begin{lemma} If~ $\ol x\in C^1([a,b],\mathbb{R}_\mathcal{I})$,  then the interval differential equation \eqref{L1.1} is equivalent to
\begin{equation} \label{L1.12.}
\ol x(t)=\ol x_0+(IR) \int_{a}^{t}\ol f(s,\ol x(s))ds,~~t\in[a,b].
\end{equation}
\end{lemma}

We are now ready to give our main results.

\begin{theorem}\label{L1.2} Suppose that $\ol f$ in   \eqref{L1.1} satisfies assumptions $(H_1)$-$(H_3)$.
If $\ol x\in C^1 ([a,b],\mathbb{R}_\mathcal{I}) $,
then there exists at least one solution of \eqref{L1.1}.
\end{theorem}
\begin{proof}
Since   $\ol g,\ol h$ are Riemann integrable on $[a,b]$,  the primitives   of $\ol g$ and $\ol h$ are continuous on $ [a,b]$.
 Thus, $\underline{G},$ $\overline{G}$, $\underline{H} $ and $\overline{H}$ are also continuous functions on $[a,b]$, so they are bounded.
Let $$M=
\max\limits_{t\in [a,b]}\left\{ \left\Vert (IR)\displaystyle\int_{a}^{t}\ol g(s)\dd s\right\Vert,~\left\Vert (IR)\displaystyle\int_{a}^{t}\ol h(s)\dd s\right\Vert\right\}.$$
According to     $(H_3)$, Lemma \ref{lem2.8}, and Proposition \ref{proprop}, one has
$$(IR)\int_{a}^{t}\ol g(s)\dd s\leq (IR) \int_{a}^{t}\ol f(s,\ol x(s))\dd s\leq (IR) \int_{a}^{t}\ol h (s)\dd s.$$
Let $E=\{ x\in C^1([a,b],\mathbb{R}_\mathcal{I}):~\Vert x-x_{0}\Vert_\infty\leq M \}$,~and  define the operator $\mathcal{T}$ by
\begin{equation}
\mathcal{T}(\ol x)(t)=\ol x_0+(IR)\displaystyle\int_{a}^{t}\ol f(s,\ol x(s))ds,~~t\in[a,b].
\end{equation}
For any $x\in E$, we have
\begin{align*}
~\Vert \mathcal{T}(\ol x)-\ol x_0\Vert_\infty
= &\sup\limits_{t\in [a,b]} d\left((IR)\displaystyle\int_{a}^{t}\ol f(s,\ol x(s))\dd s,\ol 0\right)\\
\leq& \max\left\{\sup\limits_{t\in [a,b]} d\left((IR)\displaystyle\int_{a}^{t}\ol g(s)\dd s,\ol 0\right), \sup\limits_{t\in [a,b]} d\left((IR)\displaystyle\int_{a}^{t}\ol h(s)\dd s,\ol 0\right)\right\}\\
\leq& M.
\end{align*}
 Thus,~$\mathcal{T}$ maps the set $E$ to itself,~that is,~$\mathcal{T}:E\rightarrow E$.

On the other hand,~for every $x\in E$ and $t_{1},~t_{2}\in[a,b],$~if $t_{1}\leq t_{2}$,
by $(H_3)$, we also get that
\begin{equation}
(IR) \int_{t_{1}}^{t_{2}}\ol g(s)\dd s \leq (IR) \int_{t_{1}}^{t_{2}}\ol f(s,\ol x(s))\dd s\leq
(IR) \int_{t_{1}}^{t_{2}}\ol h(s)\dd s.
\end{equation}
\\
Hence,~
\begin{equation}\label{L1.121}
\begin{split}
d(\mathcal{T}(x)(t_{2}),\mathcal{T}(x)(t_{1}))
=& d\left((IR)\displaystyle\int_{t_{1}}^{t_{2}}\ol f(s,\ol x(s))\dd s,\ol 0\right)\\
\leq&
d\left((IR)\displaystyle\int_{t_{1}}^{t_{2}}\ol g(s) \dd s,\ol 0\right)+
d\left((IR)\displaystyle\int_{t_{1}}^{t_{2}}\ol h(s) \dd s,\ol 0\right).
\end{split}
\end{equation}
Since the primitives of $g$, $h$ are continuous and so are uniformly continuous on $[a,b]$, respectively. Thus,  by  \eqref{L1.121}, we see that $\mathcal{T}(E)$ is equicontinuous.
 In view of the Ascoli-Arzel\`a Theorem, $\mathcal{T}(E)$ is relatively compact.

Finally, we only need to prove that $\mathcal{T}:E\rightarrow E$ is continuous.~Suppose that $\ol x_{m}\in E$, $m=1,2,...,$ and there exists $\ol x^{*}\in E$ such that $\ol x_{m}\rightarrow \ol x^{*}$ as $m\rightarrow \infty$. According to  $(H_1)$,
\begin{equation}
\overline{f}(\cdot, \ol x_{m})\rightarrow \overline{f}(\cdot, \ol x^{*}),~~\text{ as } m\rightarrow \infty.
\end{equation}
 By  Theorem \ref{DCT}, we have
\begin{equation}
\lim\limits_{m\rightarrow \infty}(IR)\displaystyle\int_{a}^{t}\ol f(s, \ol x_{m}(s))\dd s=
(IR)\displaystyle\int_{a}^{t}\ol f(s, \ol x^{*}(s))\dd s.
\end{equation}
This implies that $\lim\limits_{m\rightarrow \infty}\mathcal{T}(\ol x_{m})(\cdot)=\mathcal{T}(\ol x^{*})(\cdot)$, and hence $\mathcal{T}$ a continuous operator. Thus, $\mathcal{T}$ is a compact operator. By the Schauder's fixed point theorem, it follows that $\mathcal{T}$ has a fixed point in $E$, which is a solution of  \eqref{L1.1}.
\end{proof}

\begin{example} Considering the interval  differential equation
\begin{equation} \label{e1}
\begin{cases}
  \ol x'(t)=\displaystyle\frac{[1,2] t}{1+ \ol x^2(t)},&   t\in [0,4],\\
\ol x(0) = [-1,1].&
\end{cases}
\end{equation}
It is easy to see that $\ol f(t,\ol x(t))=\frac{[1,2] t}{1+ \ol x^2(t)} $  satisfies $(H_1)$ and $(H_2)$, and since $\ol 0\preccurlyeq \frac{1}{1+ \ol x^2(t)}\preccurlyeq \ol  1 $ for all $x\in C^1([0,1],\mathbb R_{\mathcal I})$ and $t\in[0,1]$,  one has
$$\ol 0\preccurlyeq  f(t,x(t)) \preccurlyeq [1,2] t , ~~ x\in C^1([0,1],\mathbb R_{\mathcal I}),  ~~ t\in[0,1],$$
i.e., $(H_3)$ holds. By Theorem \ref{L1.2}, there exists a   solution  of (\ref{e2}). Moreover, we have
$$\ol x(t)=\left[\xi(t)- \xi^{-1}(t)-\ee^{ \eta(t)-\eta^{-1}(t)},
~\xi(t)- \xi^{-1}(t)+\ee^{ \eta(t)-\eta^{-1}(t)}\right],\qquad t\in[0,1],$$
where
$$\xi(t)=\frac12\left(9 t^{2}+\sqrt{81 t^{4}+64}\right)^{\frac13},$$
$$\eta(t)=\frac12\left(-6   t^{2} \ln   2+2 \sqrt{16+9 t^{4} \ln^{2} 2}\right)^{\frac13},$$
see Figure \ref{ex64} .
\end{example}

\begin{figure}[htp]
  \centering
\subfigure[{Solution of \eqref{e1}}]{
\includegraphics[width=0.47\linewidth]{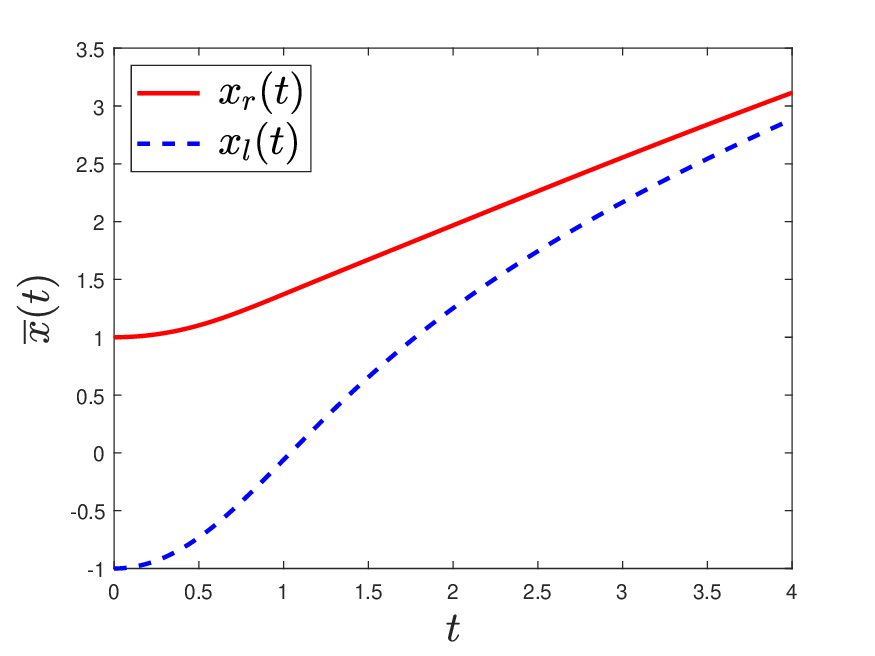}
\label{ex64}
}
  \subfigure[{Solution of \eqref{e1-1}}]{
\includegraphics[width=0.47\linewidth]{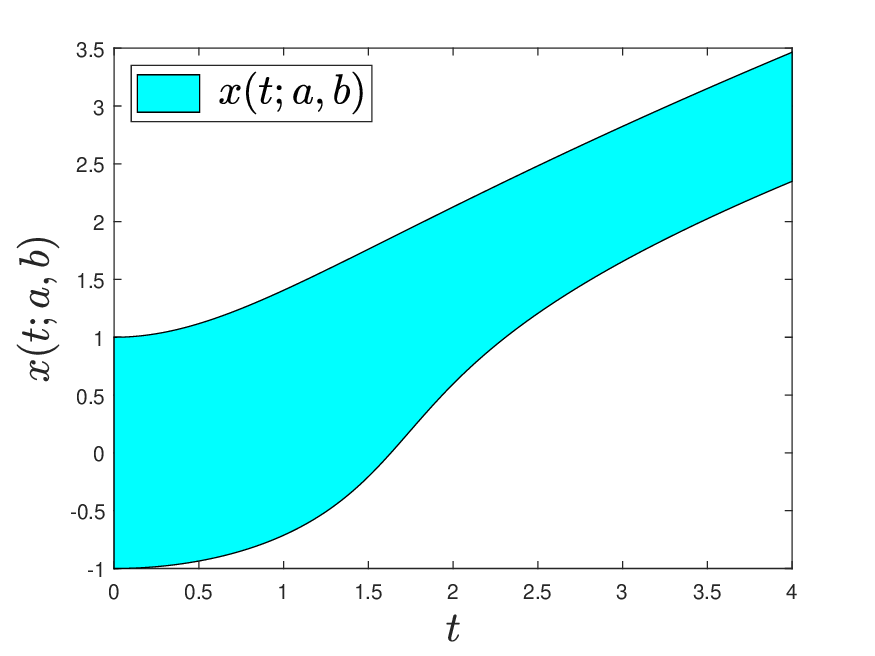}
\label{ex65}
}
   \caption{Solutions of \eqref{e1} and \eqref{e1-1}}\label{ex645}
\end{figure}
\begin{example}
  Considering the  following classical  differential equation with parameters
\begin{equation} \label{e1-1}
\begin{cases}
    x'(t)=\displaystyle\frac{b t}{1+   x^2(t)},&   t\in [0,4],~b\in [1,2],\\
  x(0) = a,& a\in [-1,1].
\end{cases}
\end{equation}
It is easy to see that the solution of  \eqref{e1-1} is
$$x(t;a,b)=   y(t;a,b)-y^{-1}(t;a,b) ,~~~ t\in [0,4],~ a\in [-1,1],~b\in [1,2],$$
where
$$y(t;a,b)=\frac12\left(6 b \,t^{2}+4 a^{3}+12 a +2 \sqrt{
9 b^{2} t^{4}+(12 a^{3} b +36 a b)t^{2}+4 a^{6}+24 a^{4}+ 36 a^{2}+16}\right)^{\frac13},$$
see Figure \ref{ex65}.

In fact, \eqref{e1-1} can be regarded as the parameter differential equation of \eqref{e1}, and from Figure \ref{ex645}, it can be seen that the two solutions are roughly consistent. Therefore, the presented algorithm is effective.
\end{example}

\begin{example}
Considering the interval  differential equation
\begin{equation} \label{e2}
\begin{cases}
 \ol  x'(t)=\displaystyle\frac{[1,2]\sin t}{1+ \ol x^2(t)} ,&   t\in [0,4],\\
\ol x(0) = [1,3].&
\end{cases}
\end{equation}
It is easy to see that
$$\ol x(t)=\left[\xi(t)- \xi^{-1}(t)-\ee^{ \eta(t)-\eta^{-1}(t)},
~\xi(t)- \xi^{-1}(t)+\ee^{ \eta(t)-\eta^{-1}(t)}\right],\qquad t\in[0,1],$$
where
$$\xi(t)=\frac12\bigg(-18 \cos t+74+2 \sqrt{81\cos^{2}t-666 \cos t+1385}\bigg)^{\frac13},$$
$$\eta(t)=\frac12\bigg(4 \sqrt{4+9 \left(\cos  t -1\right)^{2} \ln ^{2} 2}+12 \left(\cos t -1\right) \ln 2\bigg)^{\frac13},$$
see  Figure \ref{ex66}.
\end{example}

\begin{figure}[htp]
  \centering
\subfigure[{Solution of \eqref{e2}}]{
\includegraphics[width=0.47\linewidth]{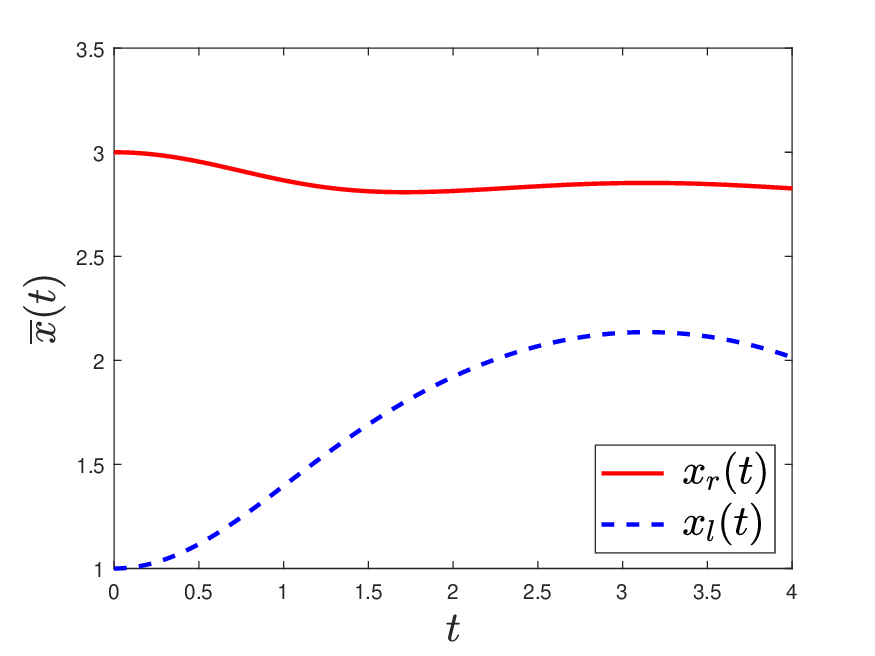}
\label{ex66}
}
  \subfigure[{Solution of \eqref{e2-1}}]{
\includegraphics[width=0.47\linewidth]{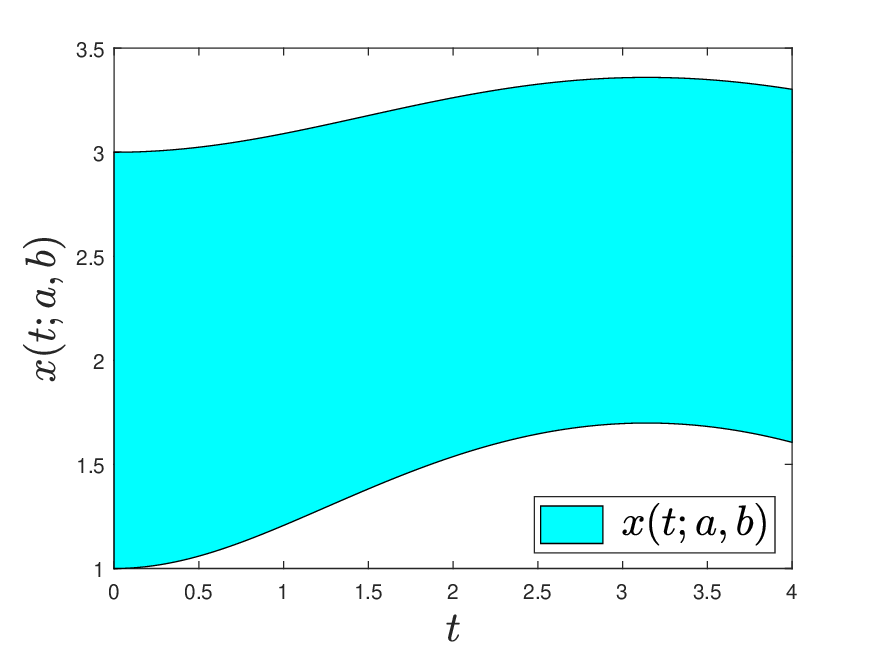}
\label{ex67}
}
   \caption{Solutions of \eqref{e2} and \eqref{e2-1}}\label{ex667}
\end{figure}

\begin{example}
  Considering the  following classical  differential equation with parameters
\begin{equation} \label{e2-1}
\begin{cases}
    x'(t)=\displaystyle\frac{b \sin t}{1+   x^2(t)},&   t\in [0,4],~b\in [1,2],\\
  x(0) = a,& a\in [ 1,3].
\end{cases}
\end{equation}
It is easy to see that
$$x(t;a,b)=  y(t;a,b) -  y^{-1}(t;a,b) ,~~~  t\in [0,4],~b\in [1,2],~ a\in [ 1,3],$$
where
$$y(t;a,b)=\frac12\left(4 a^{3}+12 a +12 b \left(1-\cos \! \left(t \right)\right)+4 \sqrt{\left(a^{3}+3 a +3 b \left(1-\cos \! \left(t \right)\right)\right)^{2}+4}\right)^{\frac13},$$
see Figure \ref{ex67}.

As showning in Figure \ref{ex667}, it can be seen that the two solutions of  \eqref{e2} and \eqref{e2-1} are roughly consistent.

\end{example}

Now we present two linear interval differential equations to show the differences between the new derivative and the $gH$-derivative.
\begin{example}
  Considering the  following  interval  differential equation
\begin{equation} \label{e3}
\begin{cases}
  \ol  x'(t)= -\ol x(t)+[1,2]t,&   t\in [0,3],\\
 \ol x(0) = [0,1].&
\end{cases}
\end{equation}
The corresponding parameter differential equation is given by
 \begin{equation} \label{e3-1}
\begin{cases}
     x'(t)= -  x(t)+b t,&   t\in [0,3],~ b\in [1,2],\\
 \ol x(0) =a, &a\in [0,1].
\end{cases}
\end{equation}
The corresponding interval differential equation under the $gH$-derivative is given by
 \begin{equation} \label{e3-2}
\begin{cases}
  \ol  x'_{gH}(t)= -1\odot\ol x(t)+[1,2]\odot t,&   t\in [0,3],\\
 \ol x(0) = [0,1].&
\end{cases}
\end{equation}
 The solution of \eqref{e3} is
\begin{align*}
  \ol x(t)=&[x_l(t),x_r(t)]\\
=&\left[2\ee^{-t}+\frac32(t-1)-\ee^{-2\ee^{-t}\ln 2-(t-1)\ln 2},2\ee^{-t}+\frac32(t-1)+\ee^{-2\ee^{-t}\ln 2-(t-1)\ln 2}\right],~t\in[0,3].
\end{align*}
 The solution of \eqref{e3-1} is
$$x(t;a,b)= \ee^{-t} \left(a +b \right)+b \left(t -1\right), ~t\in[0,3], ~a\in [0,1],~b\in [1,2].
$$
 The solutions of \eqref{e3-2} are\begin{align*}
   \ol x_{1}(t)&=[(x_1)_l(t),(x_1)_r(t)]=\left[2t-\ee^t+2\ee^{-t}-1, t+\ee^t+2\ee^{-t}-2\right], \\
 \ol x_{2}(t)&=[(x_2)_l(t),(x_2)_r(t)],
 \end{align*}
where
$$(x_2)_l(t)=\begin{cases}
  2t+2\ee^{-t}-2, &\text{ if } 0\leq t\leq 1,\\
  2t-\ee^{t-1}+2\ee^{-t}-1, &\text{ if } 1\leq t\leq 3,
\end{cases}$$$$ (x_2)_r(t)=\begin{cases}
   t+2\ee^{-t}-1, &\text{ if } 0\leq t\leq 1,\\
   t+\ee^{t-1}+2\ee^{-t}-2, &\text{ if } 1\leq t\leq 3.
\end{cases}$$

From Figure \ref{ex68}, we can see that the solution of \eqref{e3} is more close to that of \eqref{e3-1} than that of \eqref{e3-2}.
Moreover, as $t$ increases, the width of $\ol x_1$ and $\ol x_2$ gradually increases, and the rate of increase becomes faster and faster, which is inconsistent with the solution of \eqref{e3-1}.
\begin{figure}[ht]
  \centering
  \includegraphics[width=17cm]{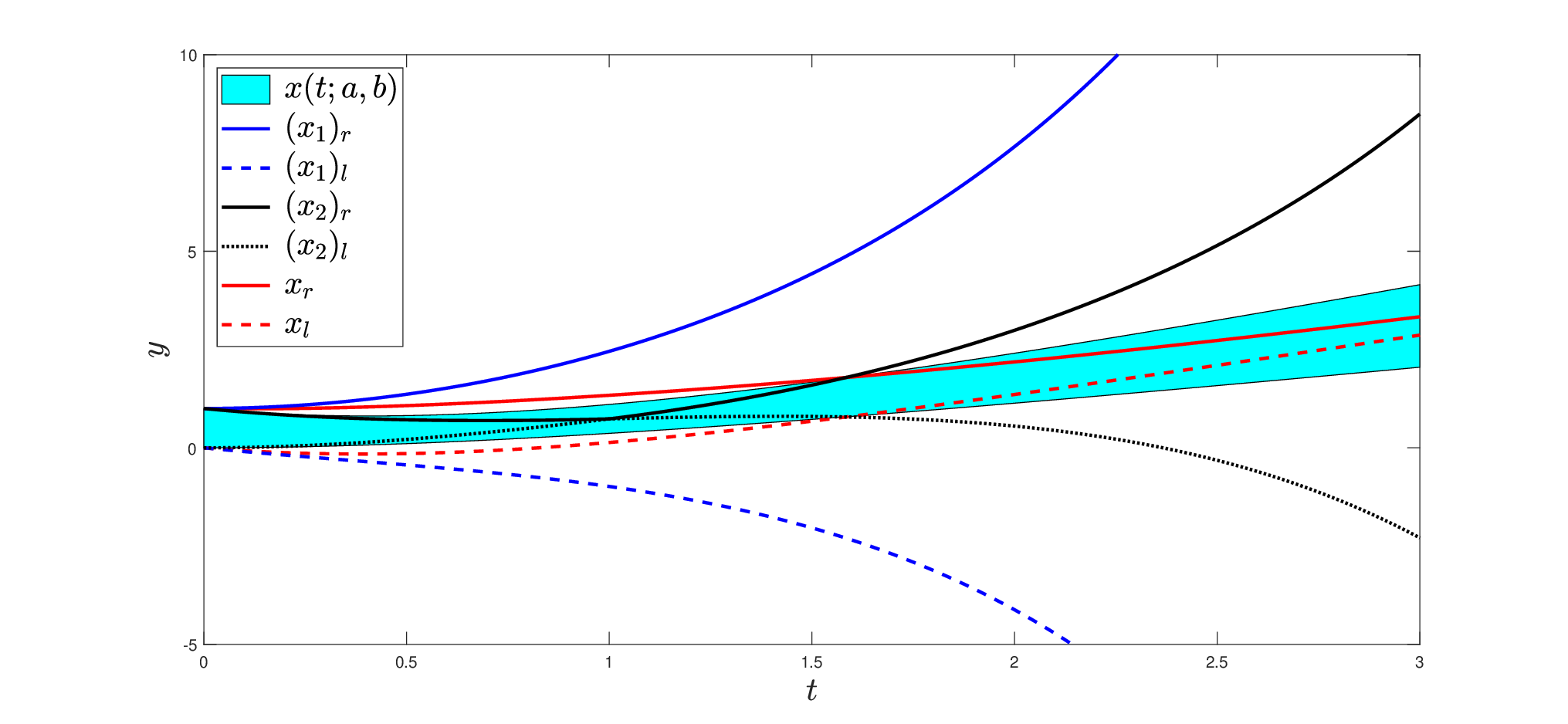}
  \caption{Solutions  of \eqref{e3}--\eqref{e3-2}}\label{ex68}
\end{figure}

\end{example}

\begin{example}
Considering the interval  differential equation
\begin{equation} \label{e4}
\begin{cases}
 \ol x'(t)= \ol x(t) \sin t,&   t\in [0,6],\\
\ol x(0) = [1,2].&
\end{cases}
\end{equation}
 The corresponding parameter differential equation is given by
 \begin{equation} \label{e4-1}
\begin{cases}
     x'(t)=    x(t) \sin t,&   t\in [0,6],\\
 \ol x(0) =a, &a\in [1,2].
\end{cases}
\end{equation}
The corresponding interval differential equation under the $gH$-derivative is given by
 \begin{equation} \label{e4-2}
\begin{cases}
  \ol  x'_{gH}(t)=  \ol x(t) \odot \sin t,&   t\in [0,6],\\
 \ol x(0) = [1,2].&
\end{cases}
\end{equation}
 The solution of \eqref{e4} is
$$\ol x(t)=[x_l(t),x_r(t)]=\left[\frac32\ee^{1-\cos t} -2^{-\ee^{1-\cos t} },\frac32\ee^{1-\cos t}+2^{-\ee^{1-\cos t} }\right],~t\in[0,6].$$
 The solution of \eqref{e4-1} is
$$x(t;a )= a\ee^{1-\cos t}  , ~t\in[0,6], ~a\in   [1,2].
$$
 The solutions of \eqref{e4-2} are\begin{align*}
   \ol x_{1}(t)&=[(x_1)_l(t),(x_1)_r(t)]=\left[\ee^{1-\cos t}, 2\ee^{1-\cos t}\right], \\
 \ol x_{2}(t)&=[(x_2)_l(t),(x_2)_r(t)]=\left[\frac32\ee^{1-\cos t}-\frac12\ee^{\cos t-1},\frac32\ee^{1-\cos t}+\frac12\ee^{\cos t-1}\right],\\
 \ol x_{3}(t)&=[(x_3)_l(t),(x_3)_r(t)],\\
 \ol x_{4}(t)&=[(x_4)_l(t),(x_4)_r(t)],
 \end{align*}
where
$$(x_3)_l(t)=\begin{cases}
  \ee^{1-\cos t},& t\in[0,\pi],\\
    \frac32\ee^{1-\cos t}-\frac12\ee^{3+\cos t},& t\in[\pi,6],\\
\end{cases}
 $$
$$ (x_3)_r(t)=\begin{cases}
  2\ee^{1-\cos t},& t\in[0,\pi],\\
   \frac32\ee^{1-\cos t}+\frac12\ee^{3+\cos t},& t\in[\pi,6],\\
\end{cases}$$
$$(x_4)_l(t)=\begin{cases}
  \frac32\ee^{1-\cos t}-\frac12\ee^{-1+\cos t},& t\in[0,\pi],\\
    \frac32\ee^{1-\cos t}-\frac12\ee^{-3-\cos t},& t\in[\pi,6],\\
\end{cases}
 $$
$$ (x_4)_r(t)=\begin{cases}
 \frac32\ee^{1-\cos t}+\frac12\ee^{-1+\cos t},& t\in[0,\pi],\\
 \frac32\ee^{1-\cos t}+\frac12\ee^{-3-\cos t},& t\in[\pi,6].\\
\end{cases}$$
From Figure \ref{ex69}, we can see that the solution $\ol x_3$   has the width issue; the latter half of $\ol x_4$  almost overlaps; $\ol x_2$   and $\ol x$ are almost identical, only $\ol x_1$ perfectly coincides with the solution $x(t,a,b)$.
However, due to the issue of switching points, as $t$ increases,  the number of solutions to  \eqref{e3-2}  will doubly increase, which makes solving the equation complicated.

\begin{figure}[H]
  \centering
  \includegraphics[width=17cm]{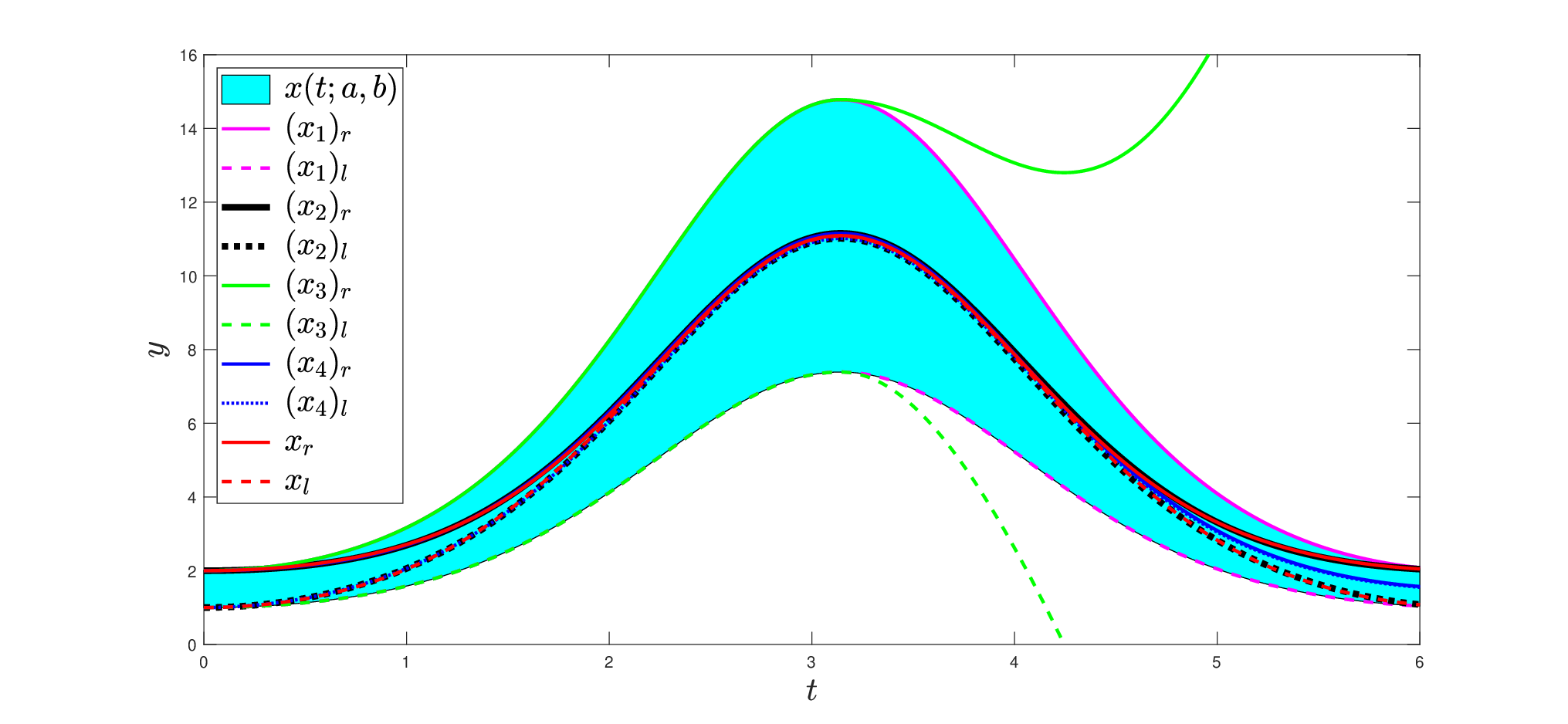}
  \caption{Solutions  of \eqref{e4}--\eqref{e4-2}}\label{ex69}
\end{figure}

\end{example}
\section{Conclusion}

In this paper, we jump out of the  original framework of interval arithmetic operations, and present new ones.  With the proposed arithmetic operations, the space $\RI$ is a linear space, while it is a semilinear space with traditional operations. Moreover, new distance, norm, and inner product are presented to show that $\RI$ is a Hilbert space. On this basis, novel definitions of derivative and integral for interval-valued functions are introduced, and their fundamental properties are systematically investigated. The proposed approach is particularly interesting due to its computational simplicity compared to existing methods. These new definitions simultaneously incorporate aspects of both the classical and multiplicative derivatives and integrals, offering a hybrid framework that retains the strengths of each. This dual-structure method not only simplifies calculations but also enhances the interpretability and applicability of the interval analysis. The classical derivative and integral are widely used in various applications. In contrast, the multiplicative derivative and integral are especially effective for modeling processes characterized by proportional or exponential growth, offering new perspectives in the analysis of interval-valued functions. Therefore, the physical interpretation of the newly introduced concepts of derivative and integral becomes more transparent and meaningful within this unified framework.

In the future, these results could be wildly applied  in other mathematical branches, such as interval integration  inequalities, interval differential equations, interval optimization,  interval decision making and so on. Also, it is easy to extend the new framework to fuzzy numbers, and the calculus of fuzzy-valued functions. These could be of interest in many fields, such as finance, economics, environment, and engineering.

\section*{Funding}

This paper was supported by ``the Fundamental Research Funds for the Central Universities" (No. B250201169).

\section*{Declaration of interests}

The authors declare that they have no known competing financial interests or
personal relationships that could have appeared to influence the work
reported in this paper.

\bibliographystyle{elsarticle-num}
\bibliography{ref}

\end{document}